\documentclass[10pt]{article}
\usepackage[utf8]{inputenc} 
\usepackage{dsfont}
\usepackage{amsmath}
\usepackage{mathtools}
\usepackage{amsthm}
\usepackage[hidelinks, unicode]{hyperref}
\usepackage{amsfonts}
\usepackage{amssymb}
\usepackage{enumitem}
\usepackage{color}
\usepackage[affil-it]{authblk}
\usepackage[style=ieee, sorting=none]{biblatex}
\usepackage{geometry}
\newtheorem{theorem}{Theorem}[section]

\newtheorem{definition}{Definition}

\newtheorem{remark}[theorem]{Remark}
\newtheorem{lemma}[theorem]{Lemma}

\newtheorem{proposition}[theorem]{Proposition}
\newtheorem{corollary}[theorem]{Corollary}
\newcommand{\R}{\mathbb{R}}

\newcommand{\N}{\mathbb{N}}
\newcommand{\Z}{\mathbb{Z}}

\newcommand{\abs}[1]{\left\vert #1\right\vert}

\newcommand{\per}{\mathrm{per}}

\newcommand{\toa}{\overset{2\alpha}{\rightharpoonup}}
\newcommand{\stoa}{\overset{2\alpha}{\rightarrow}}
\newcommand{\rightts}{\overset{2}{\rightharpoonup}}

\newcommand{\eps}{\varepsilon}

\newcommand{\Oea}{\Omega_{\eps,\alpha}}

\newcommand{\Oeaf}{\Omega_{\eps,\alpha}^{f}}
\newcommand{\Oeas}{\Omega_{\eps,\alpha}^{s}}

\newcommand{\Gea}{\Gamma_{\eps,\alpha}}

\newcommand{\Ke}{K_{\eps}}
\newcommand{\tKea}{\tilde{K}_{\eps,\alpha}}
\newcommand{\nablax}{\nabla_{\bar{x}}}

\newcommand{\ox}{\bar{x}}
\newcommand{\oy}{\bar{y}}
\newcommand{\xe}{\frac{x}{\eps}}
\newcommand{\oxe}{\frac{\ox}{\eps}}
\newcommand{\xne}{\frac{x_n}{\eps}}
\newcommand{\xnea}{\frac{x_n}{\eps^{\alpha}}}
\newcommand{\ou}{\bar{u}}

\newcommand{\cea}{c_{\varepsilon,\alpha}}
\newcommand{\tcea}{\tilde{c}_{\varepsilon,\alpha}}
\newcommand{\wea}{w_{\varepsilon,\alpha}}
\newcommand{\uea}{u_{\eps,\alpha}}
\newcommand{\Dea}{D_{\varepsilon}^{\alpha}}
\newcommand{\fea}{f_{\eps,\alpha}}
\newcommand{\feak}{f^k_{\eps,\alpha}}

\newcommand{\pea}{p_{\eps,\alpha}}
\newcommand{\vea}{v_{\varepsilon,\alpha}}
\newcommand{\hvea}{\tilde{v}_{\varepsilon,\alpha}}
\newcommand{\ceps}{c_{\varepsilon}}
\newcommand{\phiea}{\phi_{\eps,\alpha}}
\newcommand{\psiea}{\psi_{\eps,\alpha}}
\newcommand{\psieak}{\psi^k_{\eps,\alpha}}

\renewcommand{\Re}{R_{\eps}}

\newcommand{\Bea}{\mathcal{B}_{\eps,\alpha}}
\newcommand{\Sea}{S_{\eps,f}}
\newcommand{\spaceH}{\mathcal{H}}
\newcommand{\Oeah}{\Omega_{\varepsilon,\alpha,h}}
\newcommand{\bchi}{\bar{\chi}}

\newcommand{\norm}[1]{\left\Vert #1 \right\Vert}
\newcommand{\Lpnorm}[3]{\norm{#1}_{{\textstyle\mathstrut}\smash{L^{#2}(#3)}}}	
\newcommand{\Hknorm}[3]{\norm{#1}_{{\textstyle\mathstrut}\smash{H^{#2}(#3)}}}	
\newcommand{\dd}{\, \textup{d}}

\title{Homogenized limits of Stokes flow and advective transport in thin perforated domains}
\author[1]{Markus Gahn}
\author[2]{Vlad Revnic}

\affil[1]{University of Augsburg, Institute of Mathematics, Universitätsstraße 12a, 86159 Augsburg,
Germany}
\affil[2]{Heidelberg University, Institute for Mathematics, Im Neuenheimer Feld 205, 69120 Heidelberg,
Germany}
\date{}

\begin{document}

\maketitle

\begin{abstract}

We deal with the  rigorous homogenization and dimension reduction of flow and transport problems posed in thin $\eps$-periodic perforated layers with thickness of order $\eps^{\alpha}$ with $\alpha \in (0,1)$ and therefore the thickness of the layer is large compared its porosity. The aim is the derivation of effective models for $\eps\to 0 $, when the thickness of the layer tends to zero. For the flow problem we consider incompressible Stokes equations with a pressure boundary condition on the top/bottom of the layer, and the transport problem is given by reaction-diffusion-advection problem with advective flow governed from the fluid velocity from the Stokes model and different scalings for the diffusion coefficient modelling low and fast diffusion in the horizontal direction. In the limit, a Darcy-type law is obtained for the Stokes flow with Darcy-velocity depending only on the derivative of the Darcy-pressure in the vertical direction. The effective equation for the transport problem is again of diffusion-advection-type including homogenized coefficients, and with advective flow  given by the Darcy-velocity and only taking place in the vertical direction. In the case of slow diffusion in the vertical direction, effective diffusion only takes place in the vertical direction, where in the case of high diffusion in horizontal direction, we obtain effective diffusion in all space directions. To pass to the limit we use the method of two-scale convergence adapted to our microscopic geometry, which is based on uniform a priori estimates. Critical parts in the derivation of the macro-models are the control of the fluid pressure, for which we construct a Bogovskii-operator for thin perforated domains with arbitrary boundary conditions on the top/bottom, as well as the strong two-scale convergence for the microscopic solution of the transport equation, necessary to pass to the limit in the advective term. This strong convergence is established by using a Kolmogorov-Simon compactness argument.

\end{abstract}

\noindent\textbf{Keywords:} Homogenization, dimension-reduction, Stokes-equation, two-scale convergence, reaction-diffusion-advection equation

\section{Introduction}

The study of fluid flow and the transport of chemical substances or heat through thin, heterogeneous layers is crucial for numerous applications, ranging from medicine to geosciences and materials science. The different scalings in the microscopic geometry, as the thickness and the porosity of the layer leads to high computational challenges. To overcome this problem effective models for $\eps\to 0$ are derived, carrying information about the processes on the microscopic scale in homogenized coefficients.
The present work deals with the rigorous homogenization and dimension reduction of transport and flow problems posed in thin $\eps$-periodic perforated layers with thickness of order $\eps^{\alpha}$ with $\alpha \in (0,1)$. Here, $0<\eps\ll1$ is a small parameter which describes the ratio between the macroscopic size (the diameter) of the thin layer and its heterogeneity. Since $\alpha <1$, we are dealing with layers whose thickness is much greater  than their internal heterogeneity, and therefore we have a periodic structure in all space directions. However, for $\eps \to 0$ the thickness of the layer tends to zero, and therefore we are dealing with a simultaneous homogenization and dimension reduction problem. The fluid flow is described by the (quasi-) stationary incompressible Stokes equations, and the transport is given by a reaction-diffusion-advection equation, with advection given by the velocity field of the Stokes problem, and different scalings of the diffusion coefficient with respect to $\eps$ and $\alpha$. Using two-scale compactness methods, we derive for $\eps \to 0 $ limit problems on the macroscopic scale. While for the transport equation again a reaction-diffusion advection is obtained, for the Stokes problem we obtain a Darcy-type equation. 

To pass to the limit $\eps \to 0$ we make use of the two-scale convergence adapted to thin layers with thickness of order $\eps^{\alpha}$. This method captures both, the homogenization in the horizontal direction, and the dimension reduction in the vertical direction. This type of two-scale convergence was introduced in \cite{buzanic2025poroelasticplatemodelobtained}, and is an extension of the two-scale convergence from the seminal works \cite{Allaire_TwoScaleKonvergenz, Nguetseng} in domains, see also \cite{MarusicMarusicPalokaTwoScaleConvergenceThinDomains} for a first definition of two-scale convergence in thin homogeneous domains and \cite{NeussJaeger_EffectiveTransmission} for thin heterogeneous layers. Based on uniform a priori estimates for the microscopic solutions with respect to $\eps$ and $\alpha$, we obtain two-scale compactness results for these solutions. More precisely, for the fluid velocity (and pressure) and the solution of the transport equation, we get different scalings of the gradient with respect to $\eps$ and $\alpha$, leading to a different structure of the limit functions. Hence, in a first step we show general two-scale compactness results for different types of scalings of the gradient, which generalizes the results from \cite{Allaire_TwoScaleKonvergenz} to the thin layer. Although the thickness of the layer goes to zero for $\eps \to 0$, and therefore the thin layer reduces to a lower dimensional manifold, the macroscopic variable of the two-scale limit is depending on $n$ variables. In other words, the limit function is defined on a thick layer of order $1$.
This is a crucial difference compared to the case $\alpha = 1$, when the thin layer only consists of one micro-cell in the vertical direction (no periodicity in the vertical direction).

The fluid flow is described by the incompressible Stokes equations. On the top/bottom of the thin layer we impose a pressure boundary condition, and on the lateral part of the layer and the perforations inside the layer a no-slip condition is assumed. In a first step, we derive uniform a priori estimates for the fluid velocity and the fluid pressure. Here, the crucial part is the bound for the pressure. For this, we construct a Bogovskii-operator (for vector fields having arbitrary boundary values on the top/bottom of the thin layer) with suitable scalings of its operator norm with respect to $\eps $ and $\alpha$ adapted to the microscopic geometry. In a first step, we solve the divergence equation in a thin homogeneous layer with thickness of order $\eps^{\alpha}$. Now, applying the restriction operator from \cite{AllaireHomStokes} for $\eps$-periodic domains, we obtain the Bogovskii-operator for the perforated thin layer. Based on these a priori estimates and the general two-scale convergence results, we get compactness of the microscopic  Stokes-solutions. As usual, the two-scale limit of the velocity is depending on the macroscopic and the microscopic variable, while the limit pressure, the so-called Darcy-pressure, only depends on the macroscopic variable. However, in contrast to the classical case in perforated domains, see \cite{AllaireHomStokes,SanchezPalencia1980}, we only obtain $H^1$-regularity  in the $x_n$-component of the Darcy-pressure. Hence, the resulting Darcy-equation does not depend on the whole gradient of the pressure, but only on the derivative with respect to the $x_n$-component. This is also a significant difference to the case $\alpha = 1$, see for example \cite{anguiano2018transition,fabricius2023homogenization}, where the macroscopic variable for the  Darcy-pressure (and also the fluid limit) is given on a lower-dimensional manifold, and also the full gradient (with respect to the horizontal direction) contributes to the Darcy-velocity.

There is extensive literature on the homogenization of Stokes flow in perforated domains. Here, we have to mention the seminal work of Tartar in \cite{SanchezPalencia1980}, and also   \cite{AllaireHomStokes}, where a restriction operator for connected perforations is constructed. However, the homogenization and simultaneous dimension reduction of the Stokes equations posed on a thin, periodically perforated  layer has received less attention, except for  heterogeneities  of the thin layer with a specific structure, for example with a rough surface given as a graph (see \cite{bayada1989homogenization,chechkin1999boundary}) or the perforations have cylindrical shape (see \cite{anguiano2018transition}, and also \cite{fabricius2016darcy} for a formal treatment). In \cite{fabricius2023homogenization}, the case $\alpha = 1$ (only one layer of micro-cells in the vertical direction) with Navier slip boundary conditions on the perforations and the top/bottom of the thin layer is considered. Both the specific choice of $\alpha$ and the slip condition (in particular on the top/bottom) lead to qualitatively distinct effective equations compared to our problem. More precisely, the effective model only takes place on a $(n-1)$ dimensional manifold and contains an additional force term for the Darcy-velocity, due to the Navier-slip boundary condition on the perforations. Furthermore, full $H^1$-control of the macroscopic pressure is obtained.  \cite{anguiano2018transition} deals with the homogenization and dimension reduction of the Navier–Stokes equations with no-slip boundary conditions in a periodically perforated thin domain, where the periodicity scale differs from the thickness scale. Compared to the present work, for the homogenization and dimension reduction the unfolding method is used, which gives an equivalent characterization of the two-scale convergence. Furthermore, their analysis is limited to the case of cylindrical solids, without oscillations in the vertical direction. We emphasize, that this has a significant influence on the limit model. Further, in our case of a pressure boundary condition on the top/bottom of the thin layer, the a priori estimate for the pressure is of order $\eps^{\frac{3\alpha}{2}}$ instead of order $\eps^{\frac{\alpha}{2}}$ in the no-slip case. As a special case of our results for arbitrary perforations, we also consider the case of cylindrical inclusions, which leads to a Darcy-flow depending only in the vertical direction of the Darcy-pressure (and the vertical forces), where the horizontal flow is only depending on the horizontal forces multiplied with the permeability tensor.
In our paper, we use the  two-scale convergence defined in \cite{buzanic2025poroelasticplatemodelobtained}, where this method was used for the homogenization and dimension reduction of a linearized fluid-structure interaction problem coupling instationary Stokes-flow with linear elasticity for different scalings. As in our case, the thickness of the layer tends to zero for $\eps \to 0$, but the periodic oscillations within the layer are much smaller than the thickness. In the limit, a Biot-law is obtained, where the generalized Darcy-velocity is also only depending on the $n$-th derivative of the Darcy-pressure. The crucial difference is in the proof of the a priori estimates for the microscopic pressure. As usual in the derivation of the Biot-law, continuity condition between the fluid flow and the time-derivative of the displacement at the fluid-structure-interface allows the control of the pressure, while in our case we have to construct a restriction operator adapted to the microscopic geometry. We also refer to \cite{gahn2025effective} for the derivation of a Biot-plate equation in the case $\alpha = 1$.

The last part of our paper deals with the homogenization and dimension reduction of a  transport problem, modelling, for instance, the evolution of a chemical species concentration (as well as heat transfer), given by a reaction-diffusion-advection equation with advection governed by the Stokes velocity and different scaling for the diffusion coefficient depending on both $\eps$ and $\alpha$. We cover the cases of fast and slow diffusion in the horizontal direction.
On the top and bottom of the layer, we consider Dirichlet-boundary conditions and on the lateral part of the layer and on the perforations, we consider homogeneous Neumann-boundary conditions. As for the fluid flow, the first step involves deriving $\eps$-uniform  a priori estimates for the concentration. Naturally, these depend on the scaling of the diffusion coefficient. 
In order to deal with the advection term, strong two-scale convergence of the microscopic concentration is required, for which we need control of the time-derivative.  For this, additional $L^\infty$-estimates are needed. In the case of slow diffusion, standard energy bounds for the time-derivative and Sobolev norms (depending on the scaling for the diffusion) are insufficient to guarantee strong two-scale convergence of the concentration, and further control of the spatial variable is needed. This is achieved by estimating the differences of shifts of the microscopic concentration, which finally allows an application of Kolmogorov-Simon type compactness results. The different diffusion coefficients lead to two distinct limit models, as $\eps \to 0$. In the case of fast diffusion, the homogenized model exhibits effective diffusion in both the horizontal and vertical direction while advection only takes place in the vertical $x_n$-direction. It is worth emphasizing that, even though the layer reduces to a lower dimensional manifold, we still get an effect in the vertical direction. Conversely, in the case of slow diffusion, the weaker estimates only ensure diffusive and convective flow in the vertical direction.

The homogenization of reaction-diffusion-advection equations for slow and fast diffusion is nowadays well understood and we  refer to the seminal works \cite{hornung1991diffusion} and \cite{hornung1994reactive}. The latter deals in particular also with the case of slow diffusion with a specific nonlinear reaction term for the scalar case. More general nonlinearities and systems are considered in \cite{mielke2013two} and \cite{eck2005homogenization}, see also \cite{gahn2021homogenization}, where a general two-scale compactness result of Kolmogorov-Simon type is shown for problems with low diffusion.
Rigorous results for the derivation of effective models via simultaneous dimension reduction and homogenization for reaction-diffusion-problems including nonlinearities for the case $\alpha = 1$ can be found in \cite{GahnNeussRaduSingularLimit2021,GahnEffectiveTransmissionContinuous,GahnNeussRaduKnabner2018a,NeussJaeger_EffectiveTransmission} for different scalings of the diffusion coefficient (for dimension reduction problems including nonlinearities see for example \cite{chipot2011some,list2020rigorous} for different scaling of the diffusion coefficient). Perforated thin domains were considered in \cite{gahn2025derivation}. A reaction-diffusion-advection equation modelling heat flow with advective term given by the solution of a Stokes equation, was recently treated in \cite{freudenberg2024analysis} for a thin layer with rough surface given  as a graph. For $\alpha \in (0,1)$ rigorous results seem to be missing and our paper is a first essential step and we treat two critical scalings. Principal ideas to establish the strong two-scale convergence of the concentration in our transport problem are similar to those used in the aforementioned papers, in particular regarding two-scale compactness results (for thin domains) of Kolmogorov-Simon-type, which we generalized to our geometrical setting.

The paper is structured as follows.
In Section \ref{sec:MainResults}, we introduce the microscopic formulations of both the Stokes and transport problems, formulate the macroscopic models and outline the key steps in its derivation, and present the main results of our analysis. We also provide a detailed description of the underlying microscopic geometry. Section \ref{sec:two_scale_compactness} gives an introduction of the two-scale convergence adapted to thin, heterogeneses layers with thickness of order $\eps^{\alpha}$.  We further establish compactness results for $H^1-$functions, depending on different scaling of the gradient. The macroscopic models for the fluid and transport problem are derived in Section \ref{sec:fluidproblem} and \ref{sec:transport_problem} respectively. For both, we proceed in the following way: First, we establish $\eps$-uniform a priori estimates. Then, we show two-scale compactness results, and finally, we derive the macroscopic models.

\subsection{Notations}

Let $n\in \N$, then for $\Omega\subset \R^n$ a bounded Lipschitz domain, we denote by $L^p(\Omega)$ and $W^{1,p}(\Omega) $ the standard Lebesgue and Sobolev spaces with $p \in [1,\infty]$. In particular, for $p=2$, we write $H^1(\Omega)^d\coloneqq W^{1,2}(\Omega)^d$. With $S$ a subset of $\partial \Omega$, we let $H^1(\Omega,S)$ denote the $H^1(\Omega)$ functions vanishing on $S$ in the sense of traces. For norms defined on vector valued functions spaces $X^d$ with $d\in \N$, we omit the upper index and write $\|\cdot\|_X$ instead of $\|\cdot\|_{X^d}$. For a  Banach space $X$ and $p \in [1,\infty]$, we denote the usual Bochner spaces by $L^p(\Omega,X)$ and, in particular,$L^p((0, T), X)$ when time is involved. 
For the dual space of $X$, we use the notation $X'$.  The duality pairing between $X'$ and $X$ is denoted by  $\langle \cdot , \cdot \rangle_X$.

We consider the following periodic function spaces. Let $Y\coloneqq (0,1)^n$, then $C^{\infty}_{\per}(Y)$ is the space of smooth functions on $\R^n$, which are $Y$-periodic, and $H^1_{\per}(Y)$ is the closure of $C^{\infty}_{\per}(Y)$ with respect to the norm on $H^1(Y)$.  Further, for a subset $Y^{\ast}\subset Y$ with $\partial Y \subset \partial Y^{\ast}$, we denote by $H^1_{\per}(Y^{\ast})$ the space of  functions from  $H_{\per}^1(Y)$ restricted to $Y^{\ast}$. For $y \in Y$, we use the notation $\oy:= (y_1,\ldots,y_{n-1})$.

For a function $f\in H^1(\Sigma \times (a,b))$ with $a< b$ and $\Sigma \subset \R^{n-1}$ we write $\nabla_{\ox}f (x):= (\partial_1 f(x),\ldots,\partial_{n-1} f(x))$ (with $\ox:= (x_1,\ldots,x_{n-1})$ for $x \in \Omega$)  and also identify this vector in a natural way with a vector in $\R^n$ by $\nabla_{\ox} f(x) := (\nabla_{\ox} f(x),0)$. If $\Sigma $ is a rectangular domain with integer side length, we denote by $H_{\#}^1(\Sigma \times (a,b))$ the space of $\Sigma$-periodic functions in $\ox$-direction, and similar by $C_{\#}^{\infty}(\overline{\Omega})$ the space of smooth and $\Sigma$-periodic functions.
Finally, we define the  Frobenius product $B:C\coloneqq \mathrm{tr}(B^{\top} C) =  \sum_{i,j=1}^n B_{ij}C_{ij}$ for $B,C \in \R^{n\times n}$.

\section{The microscopic models and main results}\label{sec:MainResults}

In this section we briefly introduce the microscopic problems for the fluid flow and the transport problem, explain the essential steps used for the derivation of the macroscopic models and formulate the main results of this paper. The aim of this paper is twofold: First, we study a Stokes problem subject to no-slip boundary conditions on the oscillating boundary and pressure boundary conditions on the upper and lower surfaces of the thin layer. We then perform a rigorous homogenization and dimension reduction for this setting. Here, we only deal with the stationary problem. In the next step, we treat a reaction-diffusion-advection problem, where the advective velocity is given as the solution of the Stokes problem considered before (here we assume that the Stokes problem is quasi-stationary, which does not influence the previous results). For this, we assume different scalings for the diffusion coefficient, leading to a different macroscopic behavior. To pass to the limit $\eps \to 0$, we use the method of two-scale convergence adapted to thin layers of order $\eps^{\alpha}$. The different scalings lead to different bounds for (the gradient of)  the concentration and we  prove general two-scale compactness results for Sobolev functions to deal with these different cases.
\\

\subsection{The fluid problem}
Let us start with the formulation of the microscopic Stokes problem: We are looking for a fluid velocity $\uea \colon \Oeaf \to \R^n$ and a fluid pressure
$\pea \colon \Oeaf \to \R$ such that 
\begin{equation}\label{eq:Stokes_micro_strong}
   \begin{array}{rll}
      -\Delta \uea + \nabla \pea &= \fea \quad &\text{in} \;\Oeaf,  \\
      \nabla \cdot \uea &= 0 \quad &\text{in} \; \Oeaf, \\
      \uea &= 0 \quad &\text{on} \; \partial_D\Oeaf \cup \Gea, \\
      (-\nabla \uea + \pea \text{Id})\nu &= \pea^b\nu \quad &\text{on} \; \Sea^{\pm}.
    \end{array} 
\end{equation} 
Here, $f_{\eps,\alpha}$ is a volume force and $\pea^b$ a pressure boundary condition. Further, $\Oeaf$ is the  microscopic fluid domain, given as a periodic domain with perforations and with thickness of order $\eps^{\alpha}$ (we refer to Section \ref{sec:micro_domain} for more details). $\Sea^{\pm}$ describes the top/bottom of the domain and $\Gea$ the surface of the perforations. 
We emphasize that, to keep the problem a little bit simpler,  we use here the stress $(-\nabla \uea + \pea I) $ instead of $-e(\uea) + \pea I$ with the symmetric gradient $e(\uea)$, which is not strict from a physical point of view, due to the pressure boundary condition on $\Sea^{\pm}$. However, there seems to be no problem dealing with the more general case, by using the Korn inequality instead of the Poincar\'e inequality.
In a first step, we show $\eps$-uniform a priori estimates (depending additionally on the parameter $\alpha$) for the fluid velocity $\uea$ and the fluid pressure $\pea$, where as usual the critical point is the derivation of the estimate for the pressure. For this, we first show a Bogovskii-result for the whole layer $\Oea$ (without perforation), and then apply the restriction operator to obtain a Bogovskii result in the perforated layer $\Oeaf$, which allows to control the pressure. More precisely, we get
\begin{align*}
\eps^{-2} \Lpnorm{\uea}{2}{\Oeaf} + \eps^{-1} \Lpnorm{\nabla \uea}{2}{\Oeaf} + \eps^{-\alpha} \|\pea \|_{L^2(\Oeaf)} \leq  C\eps^{\alpha/2}. 
\end{align*}
Based on this estimate and general two-scale compactness results, we obtain limit functions $u_0 \in  L^2(\Omega,H_{\per}^1(Y))^n$ with $u_0 = 0$ in $Y\setminus Y_f$ and $\nabla_y \cdot u_0=0$ in $\Omega \times Y    $ ,  and $p_0 \in L^2(\Omega)$ (with $\Omega = \Sigma \times (-1,1)$ the thick layer), such that up to a subsequence (we refer to Section \ref{sec:two_scale_compactness} for the definition of the two-scale convergence)
\begin{align*}
    \eps^{-2} \uea \toa u_0, \qquad \eps^{-1} \nabla \uea \toa \nabla_y u_0 , \qquad \eps^{-\alpha} \pea \toa p_0.
\end{align*}
We emphasize that the limit fluid velocity $u_0$ is depending on both, the macroscopic variable $x\in \Omega$ and the microscopic variable $y \in Y$. In the next step, we derive a two-scale homogenized Stokes problem (see equation $\eqref{MacroscopicLimit}$), which includes all the necessary information of the limit problem. From this, we obtain that $u_0$ can be expressed as 
\begin{align}\label{eq:u0_decomosition}
     u_0(x,y) &= \sum_{i=1}^{n-1} f_0^iw_i + (f_0^n - \partial_{x_n}p_0)w_n,
\end{align}
where $(w_i,q_i)$ for $i=1,\dots,n$ are the solutions of the cell problem $\eqref{cellproblem}$. Compared to homogenization results for Stokes flow in porous media, here only the $n$-th derivative of the (Darcy) pressure $p_0$ contributes to the macroscopic fluid velocity.
We define the average of $u_0$ as the Darcy-velocity 
    \begin{equation*}
        \ou(x) \coloneqq \int_{Y_f} u_0(x,y) \dd y.
    \end{equation*}
It follows from the divergence-free condition of $\uea$ that the $n$-th component of the Darcy-velocity is constant in the $x_n$-direction, that is,  $\partial_n \bar{u}^n = 0$.  Hence,  with the permeability tensor $K$ defined in $\eqref{permabilitytensor}$, we obtain 
\begin{align*}
\begin{aligned}
        \ou &= K(f_0 - e_n \partial_{x_n} p_0)  &\text{ in }& \Omega,
        \\
        \partial_n \ou &= 0  &\text{ in }& \Omega.
\end{aligned}
\end{align*}
As usual for the homogenization of Stokes problems, we will see that the pressure boundary condition on $S_{\eps,\alpha}^{\pm}$ leads to the Dirichlet boundary condition $p_0 = p_0^b$ for the limit pressure. In total, we get the Darcy-equation
\begin{align}
\begin{aligned}\label{eq:Darcy_equation}
    \partial_{x_n} \left[ K(f_0- e_n\partial_{x_n} p_0) \right]_n &= 0  &\text{ in }& \Omega,\\
      p_0 &= p_0^b  &\text{ on }&  S^{\pm}_1.
\end{aligned}
\end{align}
In summary, we obtain the following result:
\begin{theorem}\label{thm:main_result_Stokes_flow}
Let $(\uea,\pea)$ be the weak solution of the microscopic problem $\eqref{eq:Stokes_micro_strong}$. Then, there exists $u_0 \in  L^2(\Omega,H_{\per}^1(Y))^n$ with $u_0 = 0$ in $Y\setminus Y_f$ and $\nabla_y \cdot u_0=0$ in $\Omega\times Y$,  and $p_0 \in L^2(\Omega)$, such that $\eps^{-2} \uea \toa u_0$ and $\eps^{-\alpha} \pea \toa p_0$. In addition, the Darcy-pressure $p_0$ is the unique weak solution of the Darcy-equation $\eqref{eq:Darcy_equation}$ and $u_0$ is given by formula $\eqref{eq:u0_decomosition}$.
\end{theorem}
We emphasize that in the limit $\eps\to 0$ the macroscopic quantities are given in the thick layer $\Omega = \Sigma \times (-1,1)$, although the thin layer reduces to a lower dimensional manifold $\Sigma$. Here, we have an essential different behavior compared to the case when the layer is of thickness $\eps$, where the limit functions only depend on the macroscopic variable $\ox \in \Sigma$.

\subsection{The transport problem}

In the next step, we consider the transport problem for a concentration given by a reaction-diffusion-advection equation with advection $\uea$, given as the solution of $\eqref{eq:Stokes_micro_strong}$ (now depending smoothly on time).
More precisely, we are looking for a $\cea : (0,T) \times \Oeaf \rightarrow \R$ which is the solution of
\begin{subequations}\label{eq:problem_transport_micro}
\begin{align}
\begin{aligned}
\frac{1}{\eps^{\alpha}}\partial_t \cea - \nabla \cdot (\Dea \nabla \cea - \frac{\uea}{\eps^2} \cea) &= \frac{1}{\eps^{\alpha}}g_{\eps,\alpha} &\mbox{ in }& (0,T)\times \Oeaf,
\\
-( \Dea\nabla \cea - \frac{\uea}{\eps^2} \cea) \cdot \nu &= 0 &\mbox{ on }& (0,T)\times \Gea,
\\
\cea &=  \ceps^b &\mbox{ on }& (0,T)\times (\Sea^+ \cup \Sea^-),
\\
\cea(0) &= 0 &\mbox{ in }& \Oeaf,
\end{aligned}
\end{align} 
with a source term $g_{\eps,\alpha}$ and a boundary concentration $\ceps^b$. The system is closed with suitable boundary conditions on $\partial_D \Oeaf $, which depends on the choice of the diffusion coefficient. More precisely, for the diffusion coefficient $\Dea$ we consider two different scalings with respect to $\eps$ and $\alpha$:
\begin{enumerate}[label = (D\arabic*)]
    \item\label{case:diffusion_low} $\Dea =  \eps^{\alpha}D I \in \R^{n \times n}$, 

    \item\label{case:diffusion_high} $\Dea = D \mathrm{diag}(\eps^{-\alpha},\ldots , \eps^{-\alpha},  \eps^{\alpha} ) \in \R^{n\times n}$.
\end{enumerate}
with a fixed constant $D>0$ and the unit matrix $I$ in $\R^{n\times n}$ (since it the first case the diffusion matrix $\Dea$ acts as a scalar, we will often just write $\Dea = \eps^{\alpha} D$).  On the lateral boundary we consider the following boundary condition
\begin{align}
\begin{aligned}
    -( \Dea\nabla \cea - \frac{\uea}{\eps^2} \cea) \cdot \nu &= 0 \quad\mbox{ on }(0,T)\times \partial_D \Oeaf, \, \,  \mbox{ if } \Dea = \eps^{\alpha} D,
    \\
    \cea \mbox{ is }& \Sigma\mbox{-periodic, if } \Dea = D \mathrm{diag}(\eps^{-\alpha},\ldots , \eps^{-\alpha},  \eps^{\alpha} ).
\end{aligned}
\end{align}
\end{subequations}
Hence, in the case \ref{case:diffusion_low} we consider homogeneous Neumann-boundary conditions, and in the case \ref{case:diffusion_high} periodic boundary conditions. Although we are in particular interested in the macroscopic behavior inside the domain, effects at the lateral boundary are also important for applications. We emphasize that the different choices are elemental for the derivation of the limit problem. While for $\Dea = \eps^{\alpha} D$ it would be no problem to consider periodic boundary conditions, our proof fails for Neumann-boundary conditions in the case \ref{case:diffusion_high}, see also Remark \ref{rem:estimate_shifts}.

From a physical point of view, the first case \ref{case:diffusion_low} treats slow diffusion in $\ox$-direction, where the second case \ref{case:diffusion_high} deals with fast diffusion in the horizontal direction. We will see that in the first case the diffusion in the macroscopic limit is only in the vertical direction and in the case of fast diffusion we get diffusion in all space directions.

We proceed in the same way as for the Stokes equation and first establish uniform a priori estimates with respect to $\eps$. More precisely, we get
\begin{align}\label{ineq:aux_main_results}
 \frac{1}{\eps^{\frac{\alpha}{2}}}\|\cea\|_{L^2((0,T)\times\Oeaf)} + \|\sqrt{\Dea} \nabla \cea\|_{L^2((0,T)\times\Oeaf)} \le C,
\end{align}
which immediately implies the existence of a limit function $c_0 \in L^2((0,T)\times \Omega)$, in particular independent of the microscopic variable $y$, such that up to a subsequence
\begin{align*}
    \cea \toa c_0.
\end{align*}
Further, we obtain a bound for the time-derivative on the dual space of the Sobolev space carrying the norm induced by the left-hand side of the previous inequality. For this, we need in particular an $L^{\infty}$-bound for the concentration $\cea$, to control the advective term. Using an Kolmogorov-Simon-type compactness argument, based on additional estimates for the differences of the shifts of the microscopic solutions, we can then establish also the strong two-scale convergence of the sequence $\cea$, necessary to pass to the limit $\eps \to 0$ in the advective term (since we only get the weak two-scale convergence of the fluid velocity $\uea$).

From inequality $\eqref{ineq:aux_main_results}$ we  see that the difference between the two cases lies in the scaling for the gradient $\nabla_{\ox} \cea$ with respect to the first $n-1$ components, leading to different regularity results (weak differentiability) of the limit function with respect to the spatial variable.

\noindent\textbf{The case $\Dea = D \mathrm{diag}(\eps^{-\alpha},\ldots , \eps^{-\alpha},  \eps^{\alpha} )$:}
This leads to 
\begin{align*}
     \|\nabla_{\ox} \cea\|_{L^2((0,T)\times\Oeaf)} + \eps \|\partial_n \cea\|_{L^2((0,T)\times\Oeaf)} \le C \eps^{\frac{\alpha}{2}},
\end{align*}
and we obtain  $c_0 \in H^1(\Omega)$ and additionally the existence of   corrector functions $\bar{c}_1 \in L^2((0,T)\times \Omega \times Y_f)$ with $\nabla_{\oy} \bar{c}_1 \in L^2((0,T)\times \Omega \times Y_f)^{n-1}$ and $(0,1)^{n-1}$-periodic with respect to $\oy$, and $c_1 \in L^2((0,T)\times \Omega, H_{\per}^1(0,1))$ (only depending on the $y_n$-variable), such that (up to a subsequence)
\begin{align*}
    (\nabla_{\ox} \cea , \eps^{\alpha} \partial_n \cea )\toa \nabla c_0 + (\nabla_{\oy} \bar{c}_1,\partial_{y_n} c_1).
\end{align*}

With these compactness results, we are able to pass to the limit in the weak variational equation associated to $\eqref{eq:problem_transport_micro}$. Here, we modify the standard homogenization approach based on the two-scale convergence to our thin structure. By choosing suitable test-functions, we first derive  cell problems for $\bar{c}_1$ and $c_1$, see  $\eqref{eq:cell_problem_c1_high_diffusion}$ and $\eqref{eq:cell_problem_c_1}$, which allow to express $\bar{c}_1$ and  $c_1$ in terms of $\nabla c_0$ and suitable cell solutions independent of macroscopic quantities. More precisely, we have 
\begin{align*}
\bar{c}_1 (t,x,y) = \sum_{i=1}^{n-1} \partial_{x_i} c_0(t,x) \bar{\chi}_i(y), \qquad c_1(t,x,y_n) = \partial_{x_n} c_0(t,x) \chi_n(y_n),
\end{align*}
where $\bar{\chi}_i$ for $i=1,\ldots,n-1$, and $\chi_n$ are the solutions of the cell problems   $\eqref{eq:cell_problem_diffusion}$ and $\eqref{eq:cell_problem_chi_n}$.
In the next step, we choose macroscopic test-functions, also capturing the dimension reduction to obtain with the expression of $\bar{c}_1$ and $c_1$ that $c_0$ is a unique solution of the macroscopic problem
\begin{align}
\begin{aligned}\label{eq:macro_model_transport_high_diffusion}
\partial_t c_0 - \nabla \cdot \left( D^{\ast} \nabla c_0 - c_0\ou e_n \right) &= \bar{g}_0 &\mbox{ in }& (0,T)\times \Omega,
\\
c_0 &= c_0^b &\mbox{ on }& (0,T)\times S^{\pm}_1,
\\
c_0(0) &= 0,
\\
c_0 \mbox{ is } \Sigma\mbox{-periodic}.
\end{aligned}
\end{align}
where $D^{\ast}$ is an effective diffusion coefficient, see $\eqref{def:effective_diffusion_coefficient}$, and $\ou$ is the Darcy-velocity obtained in Theorem \ref{thm:main_result_Stokes_flow}. Here, $\bar{g}_0$ is the average of the (two-scale) limit of $g_{\eps,\alpha}$.
First of all, we notice that macroscopically there is also an effect in the $x_n$-direction, although the layer reduces for $\eps \to 0$ to a lower dimensional manifold. The effective diffusion takes place in the horizontal and the vertical direction, where the advective flux only takes place in the vertical direction. Finally, let us summarize our results in the following main theorem:
\begin{theorem}\label{thm:main_result_transport_high_diffusion}
Let $\cea$ be the microscopic solution of $\eqref{eq:problem_transport_micro}$ and $\Dea = D \mathrm{diag}(\eps^{-\alpha},\ldots , \eps^{-\alpha},  \eps^{\alpha} )$. Then, there exists $c_0 \in L^2((0,T),H^1(\Omega))$ such that $\cea \stoa c_0$ and $c_0$ is the unique weak solution of the macroscopic problem $\eqref{eq:macro_model_transport_high_diffusion}$.
\end{theorem}
The proof of the compactness result, together with some additional convergence of the gradient $\nabla \cea$ can be found in Section \ref{sec:compactness_transport}, and the derivation of the macroscopic model is done in Section \ref{sec:derivation_macro_model_transport}, where we also give the definition of a weak solution of the problem $\eqref{eq:macro_model_transport_high_diffusion}$.
\\

\noindent\textbf{The case $\Dea = \eps^{\alpha} D I$:} In the case of low diffusion in horizontal direction, we obtain for the gradient $\nabla_{\ox}{\cea}$ a scaling of the form
\begin{align*}
    \eps^{\frac{\alpha}{2}} \|\nabla_{\ox} \cea \|_{L^2(\Oeaf)} \le C.
\end{align*}
In this case, we obtain no spatial regularity (differentiability) of $c_0$ with respect to $\ox$. While we obtain again the weak two-scale convergence of $\cea$ to a limit function $c_0\in L^2((0,T)\times \Omega)$, we only obtain $\partial_n c_0 \in L^2((0,T)\times \Omega)$.  Further, for the gradient we obtain the convergence
\begin{align*}
\eps^{\alpha} \nabla \cea \toa \partial_n c_0 e_n + \nabla_y c_1.
\end{align*}
Now, compared to the previous case, the scaled gradient $\nabla \cea$   does not convergence to the sum of the full gradient of $c_0$, but only the $n$-th component, and the rest is included in the gradient (with respect to $y$) of the corrector $c_1$. However, we can proceed in the same way as in the previous case, but this time we get the expression 
\begin{align*}
c_1(t,x,y) = \partial_{x_n} c_0(t,x) \chi_n(y),
\end{align*}
again, using the cell solution $\chi_n$ of  $\eqref{eq:cell_problem_diffusion}$.
Finally, the macroscopic model reads as follows:
\begin{align}
\begin{aligned}\label{eq:macro_model_transport_low_diffusion}
   \partial_t c_0 - \partial_n (D^{\ast}_{nn}\partial_n c_0 - c_0 \ou^n ) &= \bar{g}_0 &\mbox{ in }& (0,T)\times \Omega,
\\
c_0 &= c_0^b &\mbox{ on }& (0,T)\times S_1^{\pm},
\\
c_0(0) &= 0.
\end{aligned}
\end{align}
In this case, we only have diffusive and convective flow in the vertical direction. To pass to the limit in the advective term, we need again the strong (two-scale) convergence of the concentration. 
We summarize the main result in the following theorem:
\begin{theorem}\label{thm:main_result_transport_low_diffusion}
Let $\cea$ be the microscopic solution of $\eqref{eq:problem_transport_micro}$ and $\Dea = \eps^{\alpha} D$. Then, there exists $c_0 \in L^2((0,T) \times \Omega)$ with $\partial_n c_0 \in L^2((0,T)\times \Omega)$, such that $\cea \stoa c_0$ and $c_0$ is the unique weak solution of the macroscopic problem $\eqref{eq:macro_model_transport_low_diffusion}$.
\end{theorem}
For the proof of this result we again refer to Section \ref{sec:compactness_transport} and \ref{sec:derivation_macro_model_transport}.

\subsection{The microscopic domain} 
\label{sec:micro_domain}

Let $n \in \N$ with $n >2$ (for the treatment of the transport problem in Section \ref{sec:transport_problem} we will restrict this assumption to $n\le 4$) and $\Sigma \coloneqq (a,b) \subset \R^{n-1}$ with $a,b \in \Z^{n-1}$ and $a_i < b_i$ for
$i = 1,...,n-1$. Additionally, we assume that $\eps > 0$ and $\eps^{-1} \in \N $ and $\alpha \in (0,\infty)$. 
Furthermore, we assume that $\eps^{\alpha}/\eps \in \N$. This is necessary to construct the perforated layer via suitable reference cells, such that no micro-cells are intersected by the outer boundary.
We define the whole layer by 
\begin{equation*}
    \Oea \coloneqq \Sigma \times (\eps^{-\alpha},\eps^{\alpha}),
\end{equation*}
together with its top/bottom $S_{\eps}^{\pm}:= \Sigma \times \{\pm \eps^{\alpha}\}$.
Within the layer, we have a fluid part $\Oeaf$ and a solid part $\Oeas$, which have a periodical microscopic structure. 
More precisely, we define the reference cell
\begin{equation*}
    Y \coloneqq (0,1)^{n}.
\end{equation*}
The cell consists again of a fluid part $Y_f$ and a solid part $Y_s$ with a common interface
$\Gamma \coloneqq \mathrm{int}(\overline{Y_f} \cap \overline{Y_s})$. Hence, we have 
\begin{equation*}
    Y = Y_f \cup Y_s \cup \Gamma. 
\end{equation*}
We assume that $Y_f$ and $Y_s$ are open and connected with Lipschitz-boundary and fulfill
$Y_f\cap Y_s = \emptyset$.
Now, we introduce the set 
\begin{equation*}
    \Ke \coloneqq \left\{ k \in \Z^{n} \; \colon \; \eps(Y + k) \subset \Oea \right\}.
\end{equation*}
In particular, we have 
\begin{equation*}
    \Oea = \mathrm{int}\left( \bigcup_{k \in \Ke} \eps(\overline{Y}+k)\right).
\end{equation*}
We define the fluid part of the layer via 
\begin{equation*}
    \Oeaf \coloneqq \mathrm{int}\left( \bigcup_{k \in \Ke} \eps(\overline{Y}_f+k)\right).
\end{equation*}
The fluid-structure interface between the fluid and the solid part is denoted by 
\begin{equation*}
    \Gea \coloneqq \mathrm{int}\left( \bigcup_{k \in \Ke} \eps(\Gamma+k)\right).
\end{equation*}
We assume that $\Oeaf$ is a connected Lipschitz domain.
Furthermore, we denote the upper and lower part of the boundary of $\Oeaf$ by
\begin{equation*}
    \Sea^{\pm} \coloneqq \partial \Oeaf \cap (\partial\Sigma \times \{\pm\eps^{\alpha}\})
\end{equation*}
and the lateral part of $\Oeaf$ is defined by 
\begin{equation*}
    \partial_D\Oeaf\coloneqq \partial \Oeaf \cap (\partial\Sigma \times (-\eps^{\alpha},\eps^{\alpha})).
\end{equation*}
Finally, we introduce the macroscopic domain (the thick layer)
\begin{align*}
    \Omega := \Sigma \times (-1,1),
\end{align*}
which can also be obtained by rescaling the domain $\Oea$. We denote the top/bottom of $\Omega$ by
\begin{align*}
    S_1^{\pm} := \Sigma \times \{\pm1\},
\end{align*}
and the lateral boundary by
\begin{align*}
    \partial_D \Omega := \partial \Omega \setminus (S^+_1 \cup S_1^-).
\end{align*}
Here, the notation $D$ is related to the no-slip (Dirichlet) boundary condition for the fluid problem (although we consider in the transport problem  Neumann-boundary or periodic boundary conditions on this part of the boundary).

\section{The two-scale convergence for thin heterogeneous layers}
    \label{sec:two_scale_compactness}
In this section we introduce the two-scale convergence adapted to thin layers with thickness of order $\eps^{\alpha}$. Compared to the classical two-scale convergence, see \cite{Allaire_TwoScaleKonvergenz,Nguetseng}, we introduce an additional variable capturing the dimension reduction. Such a two-scale convergence was introduced in \cite{buzanic2025poroelasticplatemodelobtained}. Here, we use a slightly different notation. More precisely, in \cite{buzanic2025poroelasticplatemodelobtained} they work with the rescaled thin layer (in the fixed domain $\Omega$), where we work in the physical domain $\Oea$ (respectively later in $\Oeaf$). Our aim is the derivation of several compactness results for functions with weak derivatives with different bounds with respect to the scaling parameter $\eps$ (and $\alpha$). 
    
    \begin{definition}
        We say a sequence $\vea \in L^2(\Oea)$ converges (weakly) in the two-scale sense to a limit function $v_0 \in L^2(\Omega\times Y)$, if for all
        $\phi \in L^2(\Omega ,C_\per^{0}(Y))$ it holds that 
        \begin{equation*}
            \lim_{\eps \to 0} \frac{1}{\eps^{\alpha}} \int_{\Oea} \vea(x) \cdot \phi\left(\ox,\xnea,\oxe,\xne\right) \dd x =\int_{\Omega} \int_{Y} v_0(x,y) \cdot \phi(x,y) \dd y \dd x.
        \end{equation*}
    We write $\vea \toa v_0$.
    \\
    We say that a two-scale convergent sequence $\vea$ converges strongly in the two-scale sense, if additionally it holds that
    \begin{align*}
        \lim_{\eps\to 0}\eps^{-\frac{\alpha}{2}} \|\vea\|_{L^2(\Oea)} = \|v_0\|_{L^2(\Omega \times Y)}.
    \end{align*}
    We write $\vea \stoa v_0$.
    \end{definition}

    \begin{remark}\label{rem:two_scale_convergence}\
    \begin{enumerate}[label = (\roman*)]
        \item For $\wea \toa w_0$ it holds that
        \begin{align*}
            \|w_0\|_{L^2(\Omega\times Y)} \le \liminf_{\eps \to 0} \eps^{-\frac{\alpha}{2}} \|\wea\|_{L^2(\Oea)}.
        \end{align*}

        \item As in the usual two-scale convergence, see \cite{Allaire_TwoScaleKonvergenz} and \cite{Nguetseng}, it is straightforward to show, that a product between a strongly and weakly two-scale convergent sequence converges in the distributional sense. More precisely, we have for $\wea\in L^2(\Oea)$  and $\vea \in L^2(\Oea)$ with $\wea \stoa w_0$ and $\vea \toa v_0$, it holds for every $\phi \in C^{\infty}(\overline{\Omega})$ that
    \begin{align*}
       \lim_{\eps \to 0} \frac{1}{\eps^{\alpha}} \int_{\Oea} \wea \vea \phi\left(\ox,\frac{x_n}{\eps^{\alpha}}\right) \dd x  = \int_{\Omega} \int_Y w_0 v_0 \phi\dd y \dd x.
    \end{align*}
    In our case, we need this result for the case $w_0 \in L^2(\Omega)$ only depending on the macroscopic variable, which simpliefies the proof (no density argument for the approximation of $w_0$ is necessary).
    \end{enumerate}    
    
    \end{remark}
    
    In the following we provide several compactness results for sequences in $\uea$ in $H^1(\Oea)$ for different scalings of the gradient.  First of all, we quote the standard compactness result for suitable bounded sequences in $L^2(\Oea)$, which can be obtained by similar arguments as in the proofs of \cite{Allaire_TwoScaleKonvergenz}.

\begin{lemma}\label{lem:two_scale_compactness_basic}
For every sequence $\vea \in L^2(\Oea)$ such that
\begin{align*}
    \|\vea\|_{L^2(\Oea)} \le C\eps^{\frac{\alpha}{2}}
\end{align*}
there exists $v_0 \in L^2(\Omega \times Y)$, such that (up to a subsequence) $\vea \toa v_0$.
\end{lemma}
Now, our first compactness result for Sobolev functions treats the case when the gradient is of one $\eps$-order lower then the function itself, leading to the case that the limit function is depending on the macroscopic and the microscopic variable.    
    \begin{proposition}\label{prop:compactness_gradient_order_eps}
        Let $\vea \in H^1(\Oea)$ with 
        \begin{equation*}
            \Lpnorm{\vea}{2}{\Oea} + \eps\Lpnorm{\nabla \vea}{2}{\Oea} \leq C\eps^{\alpha/2}.
        \end{equation*}
        Then there exists $v_0 \in L^2(\Omega, H^1_\per(Y))$ such that
        \begin{equation*}
            \vea \toa v_0, \quad \nabla \vea \toa \nabla_y v_0.
        \end{equation*}
    \end{proposition}
    \begin{proof} This result was shown in
    \cite[Lemma B.4]{buzanic2025poroelasticplatemodelobtained} for the two-scale convergence on the rescaled domain $\Omega$ and in our notation in follows directly by the transformation rule.
    \end{proof}

\begin{proposition}\label{prop:compactness_v0_dn_v0}
Let $\vea \in H^1(\Oea)$ be a sequence such that
\begin{align*}
    \|\vea\|_{L^2(\Oea)} + \eps^{\alpha} \|\nabla \vea \|_{L^2(\Oea)} \le C \eps^{\frac{\alpha}{2}}.
\end{align*}
Then there exists $v_0 \in L^2(\Omega)$ with $\partial_n v_0 \in L^2(\Omega)$ and $v_1 \in L^2(\Omega,H_{\per}^1(Y)/\R)$ such that up to a subsequence
\begin{align*}
    \vea \toa v_0,\qquad \eps^{\alpha} \nabla \vea \toa \partial_n v_0 e_n + \nabla_y v_1.
\end{align*}
\end{proposition}
\begin{proof}
By the compactness result in Lemma \ref{lem:two_scale_compactness_basic} there exist $v_0 \in L^2(\Omega \times Y) $ and $\xi_0 \in L^2(\Omega \times Y)^n$ such that up to a subsequence
\begin{align*}
\vea \toa v_0, \qquad \eps^{\alpha}\nabla \vea \toa \xi_0.
\end{align*}
Since $\alpha\in (0,1)$ we obtain
\begin{align*}
    \eps\|\nabla \vea \|_{L^2(\Oea)} \le C\eps^{\frac{\alpha}{2}}
\end{align*}
and Proposition \ref{prop:compactness_gradient_order_eps} immediately implies $\nabla_y v_0 = 0$ and therefore $v_0(x,y) = v_0(x)$ is independent of $y$. Next, we show $\partial_n v_0 \in L^2(\Omega)$. Choose $\phi \in C_0^{\infty}(\Omega)$ and use the two-scale compactness of $\vea$ and $\eps^{\alpha} \nabla\vea$ to obtain
\begin{align*}
\int_{\Omega}\int_Y \xi_0^n \phi\dd y \dd x &= \lim_{\eps \to 0} \frac{1}{\eps^{\alpha}} \int_{\Oea} \eps^{\alpha} \partial_n \vea \phi\left(\ox,\frac{x_n}{\eps^{\alpha}}\right) \dd x 
\\
&= - \lim_{\eps \to 0} \frac{1}{\eps^{\alpha}} \int_{\Oea} \vea \partial_n \phi \left(\ox,\frac{x_n}{\eps^{\alpha}}\right) \dd x = - \int_{\Omega}\int_{Y} v_0 \partial_n \phi\dd y \dd x.
\end{align*}
This implies $\partial_n v_0 = \int_Y \xi_0^n\dd y$ (remember that $v_0$ is independent of $y$). It remains to identify the limit $\xi_0$. For this, we choose $\phi \in C_0^{\infty}(\Omega, C_{\per}^{\infty}(Y))^n$ such that $\nabla_y \cdot \phi = 0$ and obtain by similar arguments as above
\begin{align*}
\int_{\Omega}\int_Y \xi_0 \cdot \phi\dd y \dd x &= \lim_{\eps \to 0} \frac{1}{\eps^{\alpha}} \int_{\Oea} \eps^{\alpha} \nabla \vea \cdot \phi \left( \ox ,\frac{x_n}{\eps^{\alpha}}, \frac{x}{\eps}\right) \dd x 
\\
&= - \lim_{\eps \to 0} \frac{1}{\eps^{\alpha}} \int_{\Oea} \vea \left[ \eps^{\alpha} \nabla_{\ox} \cdot \phi + \partial_{x_n} \phi_n \right]\left( \ox ,\frac{x_n}{\eps^{\alpha}}, \frac{x}{\eps}\right) \dd x 
\\
&=-\int_{\Omega}\int_Y v_0 \partial_{x_n} \phi_n\dd y \dd x = \int_{\Omega} \int_Y \partial_n v_0 e_n \cdot \phi\dd y \dd x.
\end{align*}
By the Helmholtz-decomposition we obtain the existence of $v_1 \in L^2(\Omega,H_{\per}^1(Y)/\R) $ such that $\xi_0 = \partial_n w_0 e_n + \nabla_y v_1$, which gives the desired result.
\end{proof}
It is well-known (and can be easily shown by adapting the proof of the trace inequality from \cite{EvansPartialDifferentialEquations}), that functions $w \in L^2(\Omega)$ with weak derivative $\partial_n w \in L^2(\Omega)$ have traces on $L^2(S_1^{\pm})$. Hence, under the conditions of Proposition \ref{prop:compactness_v0_dn_v0} we obtain that $v_0$ has traces on the top/bottom $S_1^{\pm}$ of $\Omega$. The following result shows that the trace of $\vea$ on $S_{\eps}^{\pm}$ is preserved under the two-scale convergence:
\begin{proposition}\label{prop:two_scale_trace}
Under the assumptions of  \ref{prop:compactness_v0_dn_v0} it holds that 
\begin{align*}
    \|\vea\|_{L^2(S_{\eps}^{\pm})} \le C.
\end{align*}
Further, up to a subsequence it holds that $\vea|_{S_{\eps}^{\pm}} \rightts v_0|_{S_1^{\pm}}$ weakly in $L^2(\Sigma)$ (in the standard two-scale sense, see \cite{Allaire_TwoScaleKonvergenz}).
\end{proposition}
\begin{proof}
Defining $\hvea \in H^1(\Omega)$ by $ \hvea(x):= \vea(\ox,\eps^{\alpha} x_n)$. Then, a simple rescaling argument gives
\begin{align*}
\| \vea\|_{L^2(S_{\eps}^{\pm})} = \|\hvea\|_{L^2(S_1^{\pm})} &\le C \left( \|\hvea\|_{L^2(\Omega)} + \|\partial_n \hvea\|_{L^2(\Omega)} \right)
\\
&= C \left( \frac{1}{\eps^{\frac{\alpha}{2}}} \|\vea\|_{L^2(\Oea)} + \eps^{\frac{\alpha}{2}} \|\partial_n \vea\|_{L^2(\Oea)}   \right) \le C.
\end{align*}
In particular, there exists $\eta_0^{\pm} \in L^2(\Sigma \times (0,1)^{n-1})$ such that up to a subsequence $\vea|_{S_{\eps}^{\pm}} \rightts \eta_0^{\pm}$. Choosing $\phi \in C_0^{\infty}(\Omega \cup S_1^{\pm}, C_{\per}^1((0,1)^{n-1}))$ (constantly extended in $y_n$-direction), we obtain (with $\nu_n = \pm1 $ on $S_{\eps}^{\pm}$)
\begin{align*}
\int_{\Sigma} \int_{(0,1)^{n-1}}  \eta_0^{\pm}  \phi(\ox,\pm 1, \bar{y}) d\ox &= \lim_{\eps\to 0} \int_{\Sigma} \vea|_{S_{\eps}^{\pm} } \phi\left(\ox,\pm 1, \frac{\ox}{\eps}\right) d\ox 
\\
&= \lim_{\eps\to \infty} \int_{S_{\eps}^{\pm}} \vea \phi \left(\ox , \frac{x_n}{\eps^{\alpha}},\frac{\ox}{\eps}\right) d\sigma 
\\
&= \lim_{\eps \to \infty} \pm\int_{\Oea} \partial_n \vea \phi\left(\ox , \frac{x_n}{\eps^{\alpha}},\frac{\ox}{\eps}\right) + \frac{1}{\eps^{\alpha}} \vea \partial_n \phi \left(\ox , \frac{x_n}{\eps^{\alpha}},\frac{\ox}{\eps}\right)  \dd x
\\
&=  \pm \int_{\Omega} \int_Y (\partial_n v_0 + \partial_{y_n} v_1 ) \phi + v_0 \partial_n \phi\dd y \dd x = \int_{S_1^{\pm}} \int_{(0,1)^{n-1}} v_0 \phi  d\sigma,
\end{align*}
where at the end we used integration by parts, $\int_Y\partial_{y_n} v_1\dd y = 0$ and the fact that $v_0$ is independent of $y$. This gives the desired result.
\end{proof}
Next, we give a two-scale compactness result when the components of the gradient are scaled differently. We show here directly the result for the perforated domain $\Oeaf$, for which we have to use a special Helmholtz-decomposition. Let us compare the situation to the scaling in Proposition \ref{prop:compactness_v0_dn_v0}. A function $\vea \in H^1(\Oeaf)$ fulfilling the estimate in this proposition (with $\Oea$ replaced by $\Oeaf$), can be extended with the extension operator $E_{\eps}$ from Lemma \ref{lem:extension_operator} to a function $E_{\eps}\vea \in H^1(\Oea)$ fulfilling the same a priori estimate. Hence, we immediately obtain, from Proposition \ref{prop:compactness_v0_dn_v0}, that
\begin{align}\label{conv:aux_two_scale}
    \chi_{\Oeaf} \vea \toa \chi_{Y_f} v_0,\qquad \eps^{\alpha} \chi_{\Oeaf} \nabla \vea \toa \chi_{Y_f} \left( \partial_n v_0 e_n + \nabla_y v_1\right).
\end{align}
In other words, the extension operator allows to treat  the perforated layer as a homogeneous layer (this is a common approach in the homogenization theory for porous media). However, for different scalings for the gradient $\nabla_{\ox}$ with respect to the horizontal variable and the $n$-th derivative $\partial_n$, as given in the following proposition (related to the case of high diffusion in horizontal direction), such an argument is not possible, since the extension operator from Lemma \ref{lem:extension_operator} only allows to control the partial derivatives of the extended function by the full gradient of the function itself, which gives another estimate for the extended function, see also Section \ref{sec:transport_problem} for more details. We introduce the space
\begin{align*}
  H^1_{\per,\nabla_{\oy}}(Y_f):=  \{p \in L^2(Y_f) \, : \, \nabla_{\oy} p \in L^2(Y_f)^{n-1}, \, p \mbox{ is } (0,1)^{n-1}\mbox{-periodic} \},
\end{align*}
together with the norm
\begin{align*}
\|p\|_{H^1_{\per,\nabla_{\oy}}(Y_f)}^2 := \|p\|_{L^2(Y_f)}^2 + \|\nabla_{\oy} p\|_{L^2(Y_f)}^2.
\end{align*}
\begin{proposition}\label{prop:compactness_v0_nabla_v0}
Let $\vea \in H^1(\Oeaf)$ be a sequence such that
\begin{align*}
    \|\vea\|_{L^2(\Oeaf)} + \|\nabla_{\ox} \vea\|_{L^2(\Oeaf)} +  \eps^{\alpha} \|\partial_n \vea \|_{L^2(\Oeaf)} \le C \eps^{\frac{\alpha}{2}}.
\end{align*}
Then there exist $v_0 \in H^1(\Omega)$, $v_1 \in L^2(\Omega,H_{\per}^1(Y_f)/\R)$ with $\nabla_{\oy} v_1 = 0$, and $\bar{v}_1 \in L^2(\Omega,H^1_{\per,\nabla_{\oy}}(Y_f))$,  such that up to a subsequence
\begin{align*}
    \chi_{\Oeaf} \vea \toa  \chi_{Y_f} v_0,\qquad   \chi_{\Oeaf}(\nabla_{\ox} \vea , \eps^{\alpha} \partial_n \vea)  \toa \chi_{Y_f}\left(\nabla v_0 + (\nabla_{\oy} \bar{v}_1,\partial_{y_n} v_1)\right).
\end{align*}
The result is also valid for $\Oea$ instead of $\Oeaf$.
\end{proposition}
\begin{proof}
From the assumed a priori estimates on $\vea$, we get the existence of $v_0 \in L^2(\Omega \times Y)$ and $\xi_0 \in L^2(\Omega \times Y)^n$ (both vanishing in $Y_s$), such that up to a subsequence
\begin{align*}
     \chi_{\Oeaf}\vea \toa v_0, \qquad \chi_{\Oeaf} (\nabla_{\ox} \vea , \eps^{\alpha} \partial_n \vea) \toa \xi_0.
\end{align*}
Since we also have 
\begin{align*}
    \eps^{\alpha}\|\nabla \vea \|_{L^2(\Oeaf)} \le C\eps^{\frac{\alpha}{2}},
\end{align*}
we can apply Proposition \ref{prop:compactness_v0_dn_v0} (see also $\eqref{conv:aux_two_scale}$), to obtain $v_0(x,y) = \chi_{Y_f}(y) v_0(x)$ with $v_0 \in L^2(\Omega)$ such that $\partial_n v_0 \in L^2(\Omega)$, and $\xi_0^n = \chi_{Y_f}(\partial_n v_0 + \partial_{y_n} v_1 )$ for some $v_1 \in L^2(\Omega,H_{\per}^1(Y)/\R)$. Since also $\eps^{\alpha} \nabla_{\ox} \vea \toa 0$, we have $\nabla_{\oy} v_1 = 0$ in $Y_f$, and therefore $v_1$ is independent of $\oy$. Now, for all $\phi \in C_0^{\infty}(\Omega,C_{\per}^{\infty}(Y))^{n-1}$ with $\nabla_{\oy} \cdot \phi = 0$ we have (with $\bar{\xi}_0:= (\xi_0^1,\ldots,\xi_0^{n-1})$)
\begin{align*}
\int_{\Omega} \int_{Y_f} \bar{\xi}_0 \cdot \phi\dd y \dd x &= \lim_{\eps \to 0} \frac{1}{\eps^{\alpha}} \int_{\Oeaf} \nabla_{\ox} \vea \cdot \phi \left(\ox,\frac{x_n}{\eps^{\alpha}},\frac{x}{\eps}\right) \dd x
\\
&=- \lim_{\eps \to 0} \frac{1}{\eps^{\alpha}} \int_{\Oeaf} \vea \nabla_{\ox} \cdot \phi \left(\ox,\frac{x_n}{\eps^{\alpha}},\frac{x}{\eps}\right) \dd x = - \int_{\Omega} \int_{Y_f} v_0 \nabla_{\ox} \cdot \phi\dd y  \dd x.
\end{align*}
If we first choose $\phi $ constant with respect to $y$ and use that $v_0$ is independent of $y$, we get $v_0 \in H^1(\Omega)$. Then, with $\phi$ arbitrary (still $\nabla_{\oy} \cdot \phi = 0$), we get
\begin{align*}
    \int_{\Omega}\int_{Y_f} (\bar{\xi}_0 - \nabla_{\ox} v_0) \cdot \phi\dd y  \dd x = 0,
\end{align*}
which implies again by the Helmholtz decomposition below the existence of  $\bar{v}_1 \in L^2(\Omega \times Y_f)$ and $\nabla_{\oy} \bar{v}_1 \in L^2(\Omega \times Y_f)^{n-1}$ (unique up to an $L^2$-function only depending on $y_n$), such that $\bar{\xi}_0 = \nabla_{\ox} v_0 + \nabla_{\oy} \bar{v}_1$.

We define the space 
\begin{align*}
    L_{\sigma,\per}:= \left\{ u \in L^2(Y_f)^{n-1} \, : \, \int_{Y_f} u \cdot \nabla_{\oy} \phi\dd y  = 0 \,\, \mbox{for all } \phi \in C_{\per}^{\infty}(\overline{Y_f})   \right\}.
\end{align*}
Since $L_{\sigma,\per}$ is closed, we get $L^2(Y_f)^{n-1} = L_{\sigma,\per} \perp  L_{\sigma,\per}^{\perp} $. 
Obviously, we have $H^1_{\per,\nabla_{\oy}}(Y_f)\subset L_{\sigma,\per}^{\perp}$, since $C_{\per}^{\infty}(\overline{Y_f})$ is dense in $H^1_{\per,\nabla_{\oy}}(Y_f)$. Next, we define the quotient space $\widetilde{H}:= H^1_{\per,\nabla_{\oy}}(Y_f)/\ker(\nabla_{\oy})$. Now, for given $v \in L_{\sigma,\per} \subset L^2(Y_f)^{n-1}$, we consider the problem
\begin{align*}
    \widetilde{a}([p],[\phi]):= \int_{Y_f} \nabla_{\oy} p \cdot \nabla_{\oy} \phi\dd y  = \int_{Y_f} v \cdot \nabla_{\oy} \phi\dd y  =: l([\phi])
\end{align*}
for every $[p],[\phi] \in \widetilde{H}$ and $p\in [p], \, \phi \in [\phi]$. This problem is well-defined, since the kernel of $\nabla_{\oy}$ consists of $L^2$-functions only depending on $y_n$. By the Lax-Milgram Lemma we obtain the existence of a unique solution $[p]\in \tilde{H}$, and therefore the existence of $p \in H^1_{\per,\nabla_{\oy}}(Y_f)$ (unique up to a function depending on only on $y_n$), such that 
\begin{align*}
    \int_{Y_f} (\nabla_{\oy} p - v) \cdot \nabla_{\oy} \phi\dd y  = 0,
\end{align*}
for all $\phi \in H^1_{\per,\nabla_{\oy}}(Y_f)$ (in particular $\phi \in C_{\per}^{\infty}(\overline{Y_f})$). Hence, $\nabla_{\oy} p - v \in L_{\sigma,\per}$ and since $\nabla_{\oy}p,v\in L_{\sigma,\per}^{\perp}$, we obtain $\nabla_{\oy} p - v \in L_{\sigma,\per}\cap L_{\sigma,\per}^{\per} = \{0\}$, which implies 
\begin{align*}
    L^2(Y_f)^{n-1} = L_{\sigma,\per} \perp \nabla_{\oy}H^1_{\per,\nabla_{\oy}}(Y_f).
\end{align*}
This finishes the proof.
\end{proof}
In the following, we also identify the space $\{\phi \in H^1_{\per}(Y_f)\, : \,  \nabla_{\oy} \phi = 0\}$ with the space $H^1_{\per}(0,1)$. Finally, we consider the asymptotic expansion of $\vea$ which is justified by the previous compactness results. We make the ansatz
\begin{align*}
    \vea(x) = v_0\left(\ox,\frac{x_n}{\eps^{\alpha}},\frac{x}{\eps}\right) + \eps^{1-\alpha} v_1\left(\ox,\frac{x_n}{\eps{\alpha}},\frac{x}{\eps}\right) + \eps \bar{v}_1\left(\ox,\frac{x_n}{\eps{\alpha}},\frac{x}{\eps}\right) + \ldots.
\end{align*}
Hence, we obtain for our three compactness results the following expansions:
\begin{enumerate}[label = -]
    \item Proposition \ref{prop:compactness_gradient_order_eps}: Oscillations already occur in the lowest order term and we get:
    \begin{align*}
        \vea(x) = v_0\left(\ox,\frac{x_n}{\eps^{\alpha}},\frac{x}{\eps}\right).
    \end{align*}

    \item Proposition \ref{prop:compactness_v0_dn_v0}: No oscillations in the zeroth order term. Oscillations occur in the term of order $\eps^{1 - \alpha}$:
    \begin{align*}
        \vea(x) = v_0\left(\ox,\frac{x_n}{\eps^{\alpha}}\right) + \eps^{1-\alpha} v_1\left(\ox,\frac{x_n}{\eps^{\alpha}},\frac{x}{\eps}\right).
    \end{align*}

    \item Proposition \ref{prop:compactness_v0_nabla_v0}: In this case, the gradient in the horizontal direction is of the same order as the function itself, leading to a situation, when the corrector of order $\eps^{1-\alpha}$ is independent of the horizontal microscopic variable $\oy$, and an additional corrector of order $\eps$ is necessary:
    \begin{align*}
        \vea(x) = v_0\left(\ox,\frac{x_n}{\eps^{\alpha}}\right) + \eps^{1 - \alpha} v_1\left(\ox,\frac{x_n}{\eps^{\alpha}},\frac{x_n}{\eps}\right) + \eps \bar{v}_1\left(\ox,\frac{x_n}{\eps^{\alpha}},\frac{x}{\eps}\right).
    \end{align*}
\end{enumerate}

\begin{remark}
All the results can be generalized in an obvious way to the time-dependent case. More precisely, a sequence $\vea \in L^p((0,T),L^2(\Oea))$ with $p \in [1,\infty)$ converges weakly in the two-scale sense to a limit function $v_0 \in L^p((0,T),L^2(\Omega \times Y))$, if for all $\phi \in L^p((0,T),L^2(\Omega,C_{\per}^0(Y)))$ it holds that
 \begin{equation*}
            \lim_{\eps \to 0} \frac{1}{\eps^{\alpha}} \int_0^T \int_{\Oea} \vea(x) \cdot \phi\left(t,\ox,\xnea,\oxe,\xne\right) \dd x  \dd t =\int_0^T\int_{\Omega} \int_{Y} v_0(x,y) \cdot \phi(t,x,y) \dd y \dd x \dd t.
        \end{equation*}
Our compactness results are valid for $p\in (1,\infty)$. The strong two-scale convergence can be generalized in a straightforward way.  We use the same notation as in the time-independent case. It should be clear from the context, which regularity with respect to time can be obtained for the convergence.
\end{remark}

\section{The fluid problem}\label{sec:fluidproblem}

In this section we deal with the homogenization and dimension reduction of the microscopic Stokes problem $\eqref{eq:Stokes_micro_strong}$. We show uniform a priori estimates for the fluid velocity and the fluid pressure with respect to the parameter $\eps$ and $\alpha$. Using the general compactness results from Section \ref{sec:two_scale_compactness}, we get two-scale convergence of $\uea$ and $\pea$ to suitable limit functions, which allows us to pass to the limit in the microscopic problem, by choosing test-functions adapted to the structure of the limit function.

We start with the weak formulation for the microscopic problem and state the assumptions on the data: We say that$(\uea,\pea) \in H^1(\Oeaf,\partial_D\Oeaf \cup \Gea)^n \times L^2(\Oeaf)$ is a weak solution of $\eqref{eq:Stokes_micro_strong}$, if $\nabla \cdot \uea = 0$ and for all $\phiea \in H^1(\Oeaf,\partial_D \Oeaf \cup \Gea)^n$ it holds that 
\begin{equation*}\label{weakequ}
    \int_{\Oeaf} \nabla \uea \colon \nabla \phiea \dd x - \int_{\Oeaf} \pea \nabla \cdot \phiea \dd x = \int_{\Oeaf} \fea \cdot \phiea \dd x -\int_{\Sea^{\pm}} \pea^b\nu \cdot \phiea \dd \sigma .
\end{equation*}
Under the assumption that $\pea^b \in H^1(\Oeaf)$ (see below for the assumptions on the data), we can use the divergence theorem in the last term on the right-hand side to obtain
\begin{align*}
 \int_{\Oeaf} \nabla \uea \colon \nabla \phiea \dd x - \int_{\Oeaf} (\pea - \pea^b) \nabla \cdot \phiea \dd x = \int_{\Oeaf} (\fea - \nabla \pea^b) \cdot \phiea \dd x.
\end{align*}

\noindent\textbf{Assumptions on the data:}
\begin{enumerate}[label = (S\arabic*)]
\item\label{ass:Stokes_f} For the volume force $\fea \in L^2(\Oeaf)^n$, we assume 
\begin{equation*}
    \Lpnorm{\fea}{2}{\Oeaf} \leq C\eps^{\alpha/2}.
\end{equation*}

\item\label{ass:Stokes_pb} For the boundary pressure, we assume $\pea^b \in H^1(\Oea^f)$ such that
\begin{align*}
    \|\pea^b\|_{L^2(\Oeaf)} + \eps^{\alpha} \|\nabla \pea^b\|_{L^2(\Oeaf)} \le C\eps^{\frac{3\alpha}{2}}.
\end{align*}
Further, there exists $p_0^b \in L^2(\Omega)$ with $\partial_n p_0^b \in  L^2(\Omega)$ such that $\eps^{-\alpha}\pea^b \toa p_0^b$ and $p_1^b \in L^2(\Omega,H^1_\per(Y)/\R)$ such that $\nabla \pea^b \toa \partial_n p_0^b e_n + \nabla p_1^b$ (see Section \ref{sec:two_scale_compactness} for the definition of the two-scale convergence).
\end{enumerate}

\begin{corollary}
The problem $\eqref{eq:Stokes_micro_strong}$ omits a unique weak solution.
\end{corollary}
\begin{proof}
For fixed $\eps$ this result is classical and we skip the proof.
\end{proof}

\subsection{A priori estimates for the microscopic solutions $\vea$ and $\pea$}
\label{sec:apriori_fluid}
We begin by deriving the estimates for the fluid velocity $\uea$. In order to do so, we introduce the following Poincar\'e inequality on the layer $\Oeaf$.
\begin{lemma}\label{lem:Poincare}
    Let $\vea \in H^1(\Oeaf,\Gea) $ then there exists a constant $C>0$ not depending on $\eps$ such that 
    \begin{equation*}
        \Lpnorm{\vea}{2}{\Oeaf}\leq C \eps \Lpnorm{\nabla \vea}{2}{\Oeaf}.
    \end{equation*}    
\end{lemma}
\begin{proof}
    The proof is elementary and follows by decomposing $\Oeaf$ into reference cells and then applying the Poincar\'e inequality.
\end{proof}
\noindent We are now ready to derive the estimate for the fluid velocity.
\begin{proposition}\label{prop:apriori_fluid_velocity}
    Let $\uea \in H^1(\Oeaf,\partial_D \Oeaf \cup \Gea)^n$ be the weak solution of the Stokes problem $\eqref{eq:Stokes_micro_strong}$. Then it holds 
    \begin{equation*}
        \eps^{-2} \Lpnorm{\uea}{2}{\Oeaf} + \eps^{-1} \Lpnorm{\nabla \uea}{2}{\Oeaf} \leq  C\eps^{\alpha/2}. 
    \end{equation*} 
\end{proposition}
\begin{proof}
    We test the weak formulation (\ref{weakequ}) with the solution $\uea$ to obtain 
    (the $(\pea-\pea^b)$ term vanishes since $\nabla \cdot \uea = 0$)
    \begin{align*}
        \Lpnorm{\nabla \uea}{2}{\Oeaf}^2 
        &\leq \Lpnorm{\uea}{2}{\Oeaf} \Lpnorm{\fea}{2}{\Oeaf} + \Lpnorm{\uea}{2}{\Oeaf} \Lpnorm{\nabla \pea^b}{2}{\Oeaf} \\
        &\leq C \eps^{\alpha/2+1} \Lpnorm{\nabla \uea}{2}{\Oeaf}.
    \end{align*}
    Hence, with the Poincar\'e inequality from Lemma \ref{lem:Poincare} we achieve
    \begin{equation*}
        \Lpnorm{\uea}{2}{\Oeaf} + \eps \Lpnorm{\nabla \uea}{2}{\Oeaf} \leq C\eps^{\alpha/2 + 2}. 
    \end{equation*}
\end{proof}
The estimate, of the fluid pressure $\pea$ is less elementary. The goal is to construct and estimate a Bogovskii operator on the thin perforated layer $\Oeaf$, in order to obtain test-functions $\phiea$ such that $\nabla \cdot \phiea = \pea$ in $\Oeaf$. We begin by establishing the Bogovskii operator on the whole layer $\Oea$. We use the same techniques as in \cite[Lemma~5,Step~4]{GJN}, now adapted to the layer of thickness $\eps^{\alpha}$.

\begin{proposition}\label{Bogovskiiwholelayer}
    For all $\fea \in L^2(\Oea)$ there exists $\psiea \in H^1(\Oea,\partial_D \Oea)^n$ such that 
    \begin{equation*}
        \nabla \cdot \psiea = \fea \quad \text{in} \; \Oea
    \end{equation*}
    and 
    \begin{equation*}
        \Lpnorm{\psiea}{2}{\Oea} + \eps^{\alpha} \Lpnorm{\nabla \psiea}{2}{\Oea} \leq C \eps^{\alpha} \Lpnorm{\fea}{2}{\Oea}.
    \end{equation*}
\end{proposition}
\begin{proof}
    We set 
    \begin{equation*}
        \tKea \coloneqq \left\{ k \in \Z^{n-1} \times \{0\} \; \colon \; \eps^{\alpha}(Y+k) \subset \Oea \right\}.
    \end{equation*}
    Since $\eps^{\alpha}/\eps \in \N$, we have 
    \begin{equation*}
        \Oea = \mathrm{int}\left( \bigcup_{k \in \tKea} \eps^{\alpha}(Y+k) \right).
    \end{equation*}
    For $k \in \tKea$ and $f \in L^2(\Oea)$, we define
    \begin{equation*}
        \feak \colon Y \to \R, \quad \feak = \fea(\eps^{\alpha}(x+k)).
    \end{equation*}
    With the use of the Bogovskii operator, we obtain $\psieak \in H^1(Y,\partial Y \setminus S^{\pm})^n$ with $S^+ := (0,1)^{n-1} \times \{1\}$ and $S^- := (0,1)^{n-1} \times \{0\}$, such that  
    \begin{equation*}
        \nabla \cdot \psieak = \feak, \quad \Hknorm{\psieak}{1}{Y} \leq C \Lpnorm{\feak}{2}{Y}.
    \end{equation*} 
    Now, we define 
    \begin{equation*}
        \psiea \colon \Oea \to \R^n, \quad \psiea(x) \coloneqq \eps^{\alpha}\psieak\left(\frac{x}{\eps^{\alpha}}-k\right) \quad \text{for}\; x \in \eps^{\alpha}(Y+k).
    \end{equation*}
    For $\psiea$, we have 
    \begin{equation*}
        \nabla \cdot \psiea = \fea \quad \text{in} \; \Oea
    \end{equation*}
    and
    \begin{align*}
        \Lpnorm{\nabla \psiea}{2}{\Oea}^2
        &= \sum_{k \in \tKea} \int_{\eps^{\alpha}(Y+k)} \abs{\nabla \psiea^k\left(\frac{x}{\eps^\alpha} -k \right)}^2 \dd x= \sum_{k \in \tKea} \eps^{n\alpha} \int_{Y} \abs{\nabla \psiea^k}^2 \dd y \\
        &\leq C \sum_{k \in \tKea} \eps^{n\alpha} \int_Y \abs{\feak}^2 \dd y 
        = C \Lpnorm{\fea}{2}{\Oea}.
    \end{align*}
    Hence, we obtain  with the Poincar\'e inequality
    \begin{align*}
        \|\psiea\|_{L^2(\Oea)} \le C\eps^{\alpha} \|\nabla \psiea\|_{L^2(\Oea)},
    \end{align*}
    which can be obtained in the same way as the inequality in Lemma \ref{lem:Poincare} by replacing $\eps $ with $\eps^{\alpha}$ (we use the zero boundary conditions of $\psiea$ on the lateral boundary of $\eps^{\alpha}(Y + k)$):
    \begin{equation*}
        \Lpnorm{\psiea}{2}{\Oea} + \eps^{\alpha} \Lpnorm{\nabla \psiea}{2}{\Oea} \leq \eps^{\alpha}\Lpnorm{\fea}{2}{\Oea}.
    \end{equation*}
    \end{proof}
    Now, we want to establish the Bogovskii operator on the perforated layer $\Oeaf$. This is done via the restiction operator introduced by Allaire in \cite[Theorem~2.3]{AllaireHomStokes}: 
    There exists an operator $\Re \colon H^1(\Oea, \partial_D \Oea)^n \to H^1(\Oeaf,\partial_D\Oeaf \cup \Gea)^n$ with 
    \begin{align*}
        \Re \uea = \uea \quad&\text{for all} \; \uea \in H^1(\Oea,\partial_D\Oea)^n \; \text{with} \; \uea = 0 \; \text{in} \; \Oeas, \\
        \nabla \cdot \Re \uea = \nabla \cdot \uea \quad&\text{for all} \; \uea \in H^1(\Oea,\partial_D\Oea)^n \; \text{with} \; \nabla \cdot \uea = 0 \; \text{in} \; \Oeas, \\
    \end{align*}
    and 
    \begin{equation*}
        \Lpnorm{\Re \uea}{2}{\Oeaf} + \eps \Lpnorm{\nabla \Re \uea}{2}{\Oeaf} \leq C \left( \Lpnorm{\uea}{2}{\Oeaf} + \eps\Lpnorm{\uea}{2}{\Oeaf}\right)
    \end{equation*}
    for all $\uea \in H^1(\Oea,\partial_D \Oea)^n$. Using this operator, we can restrict the Bogovskii operator, constructed in Proposition \ref{Bogovskiiwholelayer}, to the perforated layer $\Oeaf$.
    \begin{corollary}\label{cor:Bogovskii_perforated_layer}
        For $\fea \in L^2(\Oeaf)$  there exists $\phiea \in H^1(\Oeaf, \partial_D \Oeaf \cup \Gea)^n$ such that 
        \begin{equation*}
            \nabla \cdot \phiea = \fea \quad \text{in} \; \Oeaf
        \end{equation*}
        and 
        \begin{equation*}
            \Lpnorm{\phiea}{2}{\Oeaf} + \eps \Lpnorm{\nabla \phiea}{2}{\Oeaf} \leq C\eps^{\alpha} \Lpnorm{\fea}{2}{\Oeaf}.
        \end{equation*}
    \end{corollary}
    \begin{proof}
        For $\fea \in L^2(\Oeaf)$ (extended to $\Oea$ by zero) there exists $\psiea \in H^1(\Oea,\partial_D \Oea)^n$ with 
        \begin{equation*}
            \nabla \cdot\psiea = \fea \quad \text{in} \; \Oea.
        \end{equation*}
        By setting $\phiea \coloneqq \Re \psiea $, we immediately obtain 
        \begin{equation*}
            \nabla \cdot \psiea = \fea \quad \text{in} \; \Oeaf.
        \end{equation*}
        Further, we obtain 
        \begin{align*}
            \Lpnorm{\nabla \phiea}{2}{\Oeaf} 
            &= \Lpnorm{\nabla\Re\psiea}{2}{\Oeaf} \\
            &\leq C\left(\eps^{-1} \Lpnorm{\psiea}{2}{\Oea} + \Lpnorm{\nabla \psiea}{2}{\Oea} \right) \\
            & \leq C \eps^{\alpha -1} \Lpnorm{\fea}{2}{\Oea}. 
        \end{align*}
        In total, we obtain  with the Poincar\'e inequality from Lemma \ref{lem:Poincare} 
        \begin{equation*}
            \Lpnorm{\phiea}{2}{\Oeaf} + \eps \Lpnorm{\nabla \phiea}{2}{\Oeaf} \leq C \eps^{\alpha} \Lpnorm{\fea}{2}{\Oeaf}.
        \end{equation*}
    \end{proof}

\begin{remark}\label{rem:pressure}
We emphasize that the result from  Corollary \ref{cor:Bogovskii_perforated_layer} is also valid for $\alpha = 1$. Further, we obtain the existence of a Bogovskii-operator $\Bea:  H^1(\Oeaf, \partial_D \Oeaf \cup \Gea)^n \rightarrow L^2(\Oeaf)$ with $\nabla \cdot \Bea(\fea) = \fea$ and 
    \begin{align*}
        \eps\|\nabla\Bea (\fea)\|_{L^2(\Oeaf)} \le C\eps^{\alpha} \|\fea\|_{L^2(\Oeaf)}.
    \end{align*}
    Here is a crucial difference between considering a pressure boundary condition and a no-slip boundary condition on $\Sea^{\pm}$. In the latter, see for example \cite{anguiano2018transition,fabricius2023homogenization}, the Bogovskii-operator, here denoted by $\Bea^0$, has to be defined on $H^1_0(\Oeaf)^n$ and only fulfills
    \begin{align*}
        \eps \|\nabla\Bea^0(\fea)\|_{L^2(\Oeaf)} \le C \|\fea\|_{L^2(\Oeaf)}.
    \end{align*}
    Hence, as can be seen in the following proposition, the pressure estimate can be improved to order $\eps^{\frac{3\alpha}{2}}$ (instead of order $\eps^{\frac{\alpha}{2}}$).
\end{remark}

    Now, we can prove the a priori estimate for the fluid pressure 
    \begin{proposition}\label{prop:apriori_fluid_pressure}
        Let $(\uea,\pea)\in H^1(\Oeaf,\partial_D \Oeaf \cup \Gea)^n \times L^2(\Oeaf)$ be the weak solution of the Stokes-problem. Then it holds 
        \begin{equation*}
            \Lpnorm{\pea}{2}{\Oeaf} \leq C\eps^{3\alpha/2}.
        \end{equation*}
    \end{proposition}
    \begin{proof}
        We test the weak formulation with $\phiea \in H^1(\Oeaf,\partial_D \cup \Gea)^n$ such that $\nabla \cdot\phiea= \pea$, obtained via Corollary \ref{cor:Bogovskii_perforated_layer}, and get 
        \begin{align*}
            \Lpnorm{\pea}{2}{\Oeaf}^2 
            \leq& \Lpnorm{\nabla \uea}{2}{\Oeaf} \Lpnorm{\nabla \phiea}{2}{\Oeaf} + \Lpnorm{\pea^b}{2}{\Oeaf} \Lpnorm{\pea}{2}{\Oeaf} \\
            &+ \Lpnorm{\fea}{2}{\Oeaf}\Lpnorm{\phiea}{2}{\Oeaf} + \Lpnorm{\nabla \pea^b}{2}{\Oeaf}\Lpnorm{\phiea}{2}{\Oeaf} \\
            \leq& C \left(\eps^{\alpha/2+1} \eps^{\alpha-1} + \eps^{3\alpha/2}+ \eps^{\alpha/2}\eps^{\alpha} + \eps^{\alpha/2}\eps^{\alpha} \right) \Lpnorm{\pea}{2}{\Oeaf} \\
            \leq& C \eps^{3\alpha/2}\Lpnorm{\pea}{2}{\Oeaf}. 
        \end{align*}
    \end{proof}
As mentioned in Remark \ref{rem:pressure} above, in the case of no-slip boundary conditions on $\Sea^+ \cup \Sea^-$, the pressure would be of order $\eps^{\frac{\alpha}{2}}$. However, we will see later that our order is somehow optimal and allows to pass to the limit $\eps\to 0$, while for the bound of order $\eps^{\frac{\alpha}{2}}$ causes trouble, see also \cite[Section 3.6.2]{buzanic2025poroelasticplatemodelobtained}.

\subsection{Two-scale compactness for the microscopic solutions $\vea$ and $\pea$}
\label{sec:compactness_fluid}

We use the uniform a priori estimates obtained in the previous section to show compactness results for the weak microscopic solution $(\uea,\pea)$ of $\eqref{eq:Stokes_micro_strong}$. Further, we establish suitable properties obtained from the divergence-free condition of $\uea$ and the zero boundary conditions on $\Gea$.

    \begin{proposition}\label{prop:compactness_micro_solution_fluid}
        The weak solution $(\uea,\pea)$ of the Stokes-problem satisfies 
        \begin{equation*}
            \eps^{-2} \uea \toa u_0, \quad \eps^{-1} \nabla \uea \toa \nabla_y u_0, \quad \eps^{-\alpha}  \pea \toa p_0
        \end{equation*}
        with $u_0 \in L^2(\Omega, H^1_\per(Y))^n$ and $p_0 \in L^2(\Omega)$. Additionally, we have $u_0 = 0$ in $\Omega \times (Y\setminus Y_f)$.
        Further, we have $\nabla_y \cdot u_0 = 0$ and for the Darcy-velocity
        \begin{align*}
            \ou(x):= \int_{Y_f} u_0\dd y  
        \end{align*}
        we have $\partial_n \ou^n  = 0$ and therefore $\ou^n$ is constant in $x_n$-direction.
    \end{proposition}
    
    \begin{proof}
        Due to the a priori estimates of the weak solution from Proposition \ref{prop:apriori_fluid_velocity} and \ref{prop:apriori_fluid_pressure}, as well as the two-scale compactness result from Proposition \ref{prop:compactness_gradient_order_eps}, we obtain the existence of $u_0 \in L^2(\Omega, H^1_\per(Y))^n$ and $p_0 \in L^2(\Omega \times Y)$ such that
        \begin{equation*}
            \eps^{-2} \uea \toa u_0, \quad \eps^{-1} \nabla \uea \toa \nabla_y u_0, \quad \eps^{-\alpha}  \pea \toa p_0.
        \end{equation*}
        Here, we extended $\uea$ and $\pea$ by zero to the whole reference cell $Y$.
        Now, let $\phi \in C^{\infty}_0(\Omega \times (Y\setminus Y_f))^n$ extended by zero to $Y$. Then, we have 
        \begin{equation*}
            0 = \frac{1}{\eps^\alpha} \int_{\Oea\setminus\Oeaf} \eps^{-2} \uea(x) \cdot \phi\left(\ox,\xnea,\oxe,\xne\right) \dd x \overset{\eps\to 0}{\rightarrow} \int_{\Omega} \int_{Y\setminus Y_f} u_0(x,y) \cdot \phi(x,y) \dd y \dd x.
        \end{equation*}
        Hence $u_0 = 0$ in $\Omega \times (Y\setminus Y_f)$. Now, we want to show $\nabla_y \cdot u_0 = 0$. For this, we choose $\phi \in C^{\infty}_0(\Omega,C^{\infty}_\per(Y_f))$. Then, we compute 
        \begin{align*}
            0 =& \frac{1}{\eps^\alpha} \int_{\Oeaf} \eps^{-1} \nabla \cdot \uea(x) \phi\left(\ox,\xnea,\xe\right) \dd x \\
            =& \frac{1}{\eps^\alpha} \int_{\Oeaf} \eps^{-1} \uea(x) \cdot \left[ \nablax \phi\left(\ox,\xnea,\xe\right) + e_n \eps^{-\alpha} \partial_{x_n}\phi\left(\ox,\xnea,\xe\right)  \right]\dd x \\
            &+\frac{1}{\eps^\alpha} \int_{\Oeaf} \eps^{-2} \uea(x) \cdot \nabla_y \phi\left(\ox,\xnea,\xe\right) \dd x \\
            \overset{\eps \to 0}{\rightarrow}& \int_{\Omega}\int_{Y_f} u_0(x,y) \cdot \nabla_y \phi(x,y) \dd y \dd x.
        \end{align*}
        In particular, it holds $\nabla_y \cdot u_0 = 0$. 
        Now, let $\phi \in C_0^{\infty}(\Omega)$. Then, we compute
        \begin{align*}
            0 
            &= -\int_{\Oeaf} \eps^{-2} \nabla \cdot \uea(x) \phi\left(\ox,\xnea\right) \dd x \\
            &=\sum_{i=1}^{n-1} \int_{\Oeaf} \eps^{-2} \uea^i(x) \partial_i \phi\left(\ox,\xnea\right) \dd x + \frac{1}{\eps^{\alpha}}\int_{\Oeaf} \eps^{-2} \uea^n(x) \partial_n \phi\left(\ox,\xnea\right) \dd x \\
            &\overset{\eps \to 0}{\rightarrow} \int_\Omega \int_{Y_f} u_0^n(x,y) \dd y \; \partial_n \phi(x) \dd x.
        \end{align*}
        This proves $\partial_n\ou^n = 0$ and therefore $\ou^n$  is constant in $x_n$-direction.
        Lastly, we want to show that $p_0$ is in fact independent of the 
        microscopic variable i.e. $p_0 \in L^2(\Omega)$. For that, we choose 
        \begin{equation*}
            \phiea(x) = \mathrm{diag}(\eps^{\beta},...,\eps^{\beta},\eps^{\gamma})\phi\left(\ox,\xnea,\xe\right)
        \end{equation*} 
        with $\phi \in C_0^{\infty}(\Omega,C_\per^{\infty}(Y_f))^n$ and, we compute
        \begin{align*}
            \int_{\Oeaf} \nabla &\uea \colon \nabla \phiea \dd x
            \\
            =& \sum_{i,j = 1}^{n-1} \int_{\Oeaf} \partial_i\uea^j(x) \left( \eps^{\beta} \partial_{x_i} \phi^j\left(\ox,\xnea,\xe\right) + \eps^{\beta-1} \partial_{y_i} \phi^j\left(\ox,\xnea,\xe\right)\right) \dd x \\
            &+ \sum_{i=1}^{n-1} \int_{\Oeaf} \partial_i \uea^n(x) \left( \eps^{\gamma} \partial_{x_i} \phi^n\left(\ox,\xnea,\xe\right) + \eps^{\gamma-1} \partial{y_i}\phi^n\left(\ox,\xnea,\xe\right)\right) \dd x \\
            &+ \sum_{j=1}^{n-1} \int_{\Oeaf} \partial_n \uea^j(x) \left( \eps^{\beta-\alpha} \partial_{x_n} \phi^j\left(\ox,\xnea,\xe\right) + \eps^{\beta-1} \partial_{y_n} \phi^j\left(\ox,\xnea,\xe\right) \right) \dd x\\
            &+ \int_{\Oeaf} \partial_n \uea^n(x) \left( \eps^{\gamma-\alpha} \partial_{x_n} \phi^n\left(\ox,\xnea,\xe\right) + \eps^{\gamma-1}\partial_{y_n} \phi^n\left(\ox,\xnea,\xe\right) \right)\dd x,
        \end{align*}

        \begin{align*}
            \int_{\Oeaf} &(\pea-\pea^b) \nabla \cdot \phiea \dd x 
            \\
            =& \sum_{i=1}^{n-1} \int_{\Oeaf} \left(\pea - \pea^b) \right) \left( \eps^{\beta} \partial_{x_i} \phi^i\left(\ox,\xnea,\xe\right) + \eps^{\beta-1} \partial_{y_i}\phi^i\left(\ox,\xnea,\xe\right) \right) \dd x\\
            &+ \int_{\Oeaf} \left(\pea - \pea^b\right) \left( \eps^{\gamma-\alpha} \partial_{x_n} \phi^n\left(\ox,\xnea,\xe\right) + \eps^{\gamma-1} \partial_{y_n} \phi^n\left(\ox,\xnea,\xe\right) \right) \dd x,
        \end{align*}

        \begin{align*}
            \int_{\Oeaf} (\fea-\nabla\pea^b)\cdot \phiea \dd x 
            =& \sum_{i=1}^{n-1} \int_{\Oeaf} (\fea^i(x)-\partial_i\pea^b(x))\eps^{\beta} \phi^i\left(\ox,\xnea,\xe\right) \dd x \\
            &+ \int_{\Oeaf} (\fea^n(x)-\partial_n\pea^b(x))\eps^{\gamma} \phi^n\left(\ox,\xnea,\xe\right) \dd x.
        \end{align*}
        By choosing $\beta = \gamma = 1-2\alpha$ and taking the limit $\eps \to 0$, we obtain
        \begin{equation*}
            0 = \int_{\Omega} \int_{Y} (p_0(x,y)-p_0^b(x)) \nabla_y \cdot \phi(x,y) \dd y \dd x \quad \text{for all} \; \phi \in C_0^{\infty}(\Omega \times Y)^n.
        \end{equation*}
          This immediately implies that $p_0$ is independent of the microscopic variable i.e. $p_0 \in L_0^2(\Omega)$, since $p_0^b \in L^2(\Omega)$.
    \end{proof}

    \subsection{Derivation of the macroscopic model}

    For the derivation of the macroscopic model, we proceed as in the proof of Proposition \ref{prop:compactness_micro_solution_fluid} with the exception of choosing $\beta = \gamma = -\alpha$ and $\phi \in C^{\infty}_0(\overline{\Omega}\setminus \partial_D\Omega,C_\per^{\infty}(Y_f))^n$ such that  $\nabla_y \cdot \phi = 0$.
    By doing so, we obtain in the limit $\eps \to 0$ 
    \begin{align}\label{MacroscopicLimit}
        \begin{aligned}
        \int_{\Omega} \int_{Y_f} \nabla_y u_0(x,y) \colon \nabla_y \phi(x,y) \dd y \dd x 
        &- \int_{\Omega} (p_0(x) - p_0^b(x)) \partial_{x_n} \int_{Y_f} \phi^n(x,y) \dd y \dd x \\
        & = \int_\Omega \int_{Y_f} (f_0(x)-e_n\partial_{x_n}p_0^b(x) - \nabla_y p_1^b(x,y)) \cdot \phi(x,y) \dd y \dd x
        \end{aligned}
    \end{align}
    for all $\phi \in C^{\infty}_0(\overline{\Omega}\setminus \partial_D\Omega,C_\per^{\infty}(Y_f))^n$ with $\nabla_y \cdot \phi = 0$. It is easy to see, that the term including $\nabla_y p_1^b$ vanishes via integration by parts. We now show, that $\partial_{x_n} p_0$ exists and is a $L^2(\Omega)$ function. In order to do this, we find $(w_i,q_i) \in H^1_{\per}(Y_f)^n \times L^2(Y_f)/\R$ solving, in the weak sense, the equation 
    \begin{equation}\label{cellproblem}
       \begin{array}{rll}
      -\Delta w_i + \nabla q_i &= e_i \quad &\text{in} \;Y_f,  \\
      \nabla \cdot w_i &= 0 \quad &\text{in} \; Y_f, \\
      w_i &= 0 \quad &\text{on} \; \Gamma, \\
      &w_i,q_i \; \text{are $Y$-periodic,}&
    \end{array} 
    \end{equation}
    for $i=1,...,n$.
    Existence and uniqueness of a solution is standard.
    We define the permeability tensor  
    \begin{equation}\label{permabilitytensor}
        K_{ij}= \int_{Y_f} \nabla w_i \colon \nabla w_j \dd y = \int_{Y_f} e_i \cdot w_j\dd y.
    \end{equation} 
     Further, we define the test-function $\phi \in C^{\infty}_0(\overline{\Omega}\setminus \partial_D\Omega,H_\per^{1}(Y_f))^n$ via
    \begin{equation*}
        \phi(x,y) \coloneqq K_{nn}^{-1}\eta(x)w(y),
    \end{equation*}
    with $\eta \in C^\infty_0(\overline{\Omega} \setminus \partial_D \Omega)$. We see that $\nabla_y \cdot \phi = 0$ and 
    \begin{equation*}
        \int_{Y_f} \phi^n \dd y = \eta.
    \end{equation*}
    Via a density argument, we can test equation  $\eqref{MacroscopicLimit}$ with $\phi$, and obtain
    \begin{align*}
        \abs{\int_{\Omega} (p_0(x)-p_0^b(x)) \partial_{x_n} \eta(x) \dd x }
        \leq& C(u_0,f_0,p_0^b,w)  \Lpnorm{\eta}{2}{\Omega},
    \end{align*}
    so in particular, that $\partial_{x_n} p_0 \in L^2(\Omega)$ and $p_0 = p_0^b$ on $S^{\pm}_1$. Hence, we can rewrite (\ref{MacroscopicLimit}) in
    \begin{align}
    \begin{aligned}\label{bogmacrolimit}
        \int_{\Omega}\int_{Y_f} &\nabla_y u_0(x,a) \colon \nabla_y \phi(x,y) \dd y \dd y 
        \\
        &+ \int_{\Omega}  \int_{Y_f} \partial_{x_n}p_0(x) \phi(x,y) \dd y \dd x - \int_{\Omega} \int_{Y_f} f_0(x) \cdot\phi(x,y) \dd y \dd x = 0,
    \end{aligned}
    \end{align}
    for all $\phi \in C^{\infty}_0(\Omega,C_\per^{\infty}(Y_f))^n$ with $\nabla_y \cdot \phi = 0$. Here the terms containing $p_0^b$ cancel each other other out via intergation by parts and due to the zero boundary condition of $\phi$.  Through the application of the Bogovskii operator, there exits $p_1 \in L^2(\Omega,L^2_0(Y_f))$ such that 
    \begin{align*}
         \int_{\Omega}\int_{Y_f} &\nabla_y u_0(x,a) \colon \nabla_y \phi(x,y) \dd y \dd y 
         \\
         &+ \int_{\Omega}  \int_{Y_f} \partial_n p_0(x) \phi(x,y) - p_1 \nabla_y \cdot \phi(x,y) \dd y \dd x = \int_{\Omega} \int_{Y_f} f_0(x) \cdot\phi(x,y) \dd y \dd x
    \end{align*}
    for all $\phi \in C^{\infty}_0(\Omega,C_\per^{\infty}(Y_f))^n$.  In other words, $(u_0,p_0,p_1)$ solves in the weak sense the equation
    \begin{equation*}
        -\Delta_y u_0 + e_n \partial_{x_n} p_0 + \nabla_y p_1 = f_0 \quad \text{in} \; \Omega \times Y_f.
    \end{equation*}
    We rewrite this equation into the form
    \begin{equation*}
         -\Delta_y u_0 +\nabla_y p_1 = \sum_{i=1}^{n-1} e_i f_0^i + e_n(f_0^n - \partial_{x_n}p_0) \quad \text{in} \; \Omega \times Y_f.
    \end{equation*}
    Together with $\nabla_y \cdot u_0 = 0$ in $\Omega \times Y_f$ and the boundary condition $u_0 = 0$ on $\Omega \times \Gamma$. Since the equation on the left-hand side is linear and the solution $(u_0,p_0,p_1)$ is unique, we obtain 
    \begin{align*}
        u_0(x,y) &= \sum_{i=1}^{n-1} f_0^iw_i + (f_0^n - \partial_{x_n}p_0)w_n, \\
        p_1(x,y) &=  \sum_{i=1}^{n-1} f_0^i q_i + (f_0^n - \partial_{x_n}p_0)q_n.
    \end{align*}
    where $(w_i,q_i) \in H^1_\per(Y_f)^n \times L^2(Y_f)/\R$ is again the unique solution of \eqref{cellproblem}. Now, we define the Darcy-velocity 
    \begin{equation*}
        \ou(x) \coloneqq \int_{Y_f} u_0(x,y) \dd y. 
    \end{equation*}
    For $j=1,...,n$, we obtain with the permeability tensor K, defined in (\ref{permabilitytensor}), 
    \begin{equation*}
        \ou_j = \sum_{i=1}^{n-1} f_0^i \int_{Y_f} w_i \cdot e_j \dd y  + (f_0^n - \partial_{x_n}p_0) \int_{Y_f} w_n \cdot e_j \dd y =  \sum_{i=1}^{n-1} K_{ij}f_0^i \dd y +  K_{n,j}(f_0^n - \partial_{x_n}p_0) 
    \end{equation*}
    and therefore
    \begin{equation}\label{eq:darcy_equation_velocity}
        \ou = K(f_0 - e_n \partial_{x_n} p_0) \quad \text{in} \; \Omega.
    \end{equation}
    Hence the tuple $(\ou,p_0)$ satisfies 
  \begin{align*}
    \bar{u} &= K(f_0- e_n\partial_{x_n} p_0)  &\mbox{ in }& \Omega,  \\
    \partial_{x_n} \bar{u}^n &= 0  &\mbox{ in }& \Omega.
  \end{align*}
In other words, the Darcy-pressure $p_0$ solves the equation
    \begin{equation*}
    \begin{array}{rcl}
        \partial_{x_n} \left[ K(f_0- e_n\partial_{x_n} p_0) \right]_n &= 0 \quad &\text{in} \;\Omega,\\
      p_0 &= p_0^b \quad &\text{on} \; S^{\pm}_1.
    \end{array}
    \end{equation*} 
It is obvious that this problem admits a unique weak solution, as well as the problem $\eqref{MacroscopicLimit}$. In particular, this implies that all convergence results are valid for the whole sequence, which completes the proof of Theorem \ref{thm:main_result_Stokes_flow}.

\subsection{The case of cylindrical inclusions $Y_s$}
\label{sec:cylindral_inclusions}
We comment on the case that the solid inclusions $Y_s$ are given as cylinders, more precisely, we have $Y_s = Y_s' \times (0,1)$ with $Y_s' \subset (0,1)^{n-1}$ strictly included. In the past, problems in this microscopic geometry received considerable attention in the literature, see for example  \cite{anguiano2018transition} and also \cite{fabricius2016darcy} for formal results. However, in both papers, a no-slip boundary condition on the top/bottom $\Sea^{\pm}$ on the thin layer was considered, which has significant influence on the macroscopic model. In this case, it is easy to check that $\nabla_{\ox} \cdot \ou = 0$ and further,  the microscopic pressure $\pea$ is of order $\eps^{\frac{\alpha}{2}}$. In particular, this implies (as can be seen from the calculations in the proof  of Proposition \ref{prop:compactness_micro_solution_fluid}, where now we can choose test-functions with $\partial_{y_n} \phi = \partial_{\ox} \phi = 0$ and $\nabla_{\oy} \cdot \phi = 0$, and therefore $\gamma $ and $\beta$ independent of each other) that the limit pressure is only depending on $\ox$ and fulfills $\nabla_{\ox} p_0 \in L^2(\Sigma)^{n-1}$. In the case of a pressure boundary condition on $\Sea^{\pm}$, such results seem to be not possible. However, we can simplify the representation for the Darcy-velocity in $\eqref{eq:darcy_equation_velocity}$ by considering in more detail the structure of $K$. 

It is easy to check (solve a similar equation on $Y_f'$) that the cell solutions $(w_i,q_i)$ for $i=1,\ldots,n-1$ are constant with respect to $y_n$ and we have $w_i^n = 0$. In particular, we get for $i=1,\ldots,n-1$ that
\begin{align*}
K_{in} = K_{ni} = \int_{Y_f} \nabla_y w_i : \nabla_y w_n\dd y  = \int_{Y_f} w_i^n\dd y  = 0,
\end{align*}
and $K$ has the block structure
\begin{align*}
    K = \begin{pmatrix}
        \bar{K} & 0 \\ 0 & K_{nn}
    \end{pmatrix},
\end{align*}
where $\bar{K}$ is the submatrix of $K$ consisting of the first $(n-1)$ columns and rows. This leads to ($\bar{f}_0  = (f_0^1,\ldots,f_0^{n-1})$) 
\begin{align*}
    \ou = (\bar{K} \bar{f}_0,0)^T + (0, K_{nn} (f_0^n - \partial_{x_n} p_0))^T.
\end{align*}
In particular, the horizontal part of the Darcy-velocity is just given by $\bar{K} \bar{f}_0$ and only the vertical velocity depends on the Darcy-pressure.

\section{The transport problem}
\label{sec:transport_problem}

Now, we deal with with the simultaneous homogenization and dimension reduction for a reaction-diffusion-advection equation $\eqref{eq:problem_transport_micro}$, where the advective velocity is obtained via the Stokes problem $\eqref{eq:Stokes_micro_strong}$, which we assume now to be quasi-stationary.   More precisely, we assume $f_{\eps,\alpha} \in L^{\infty}((0,T),L^2(\Oeaf))^n$ and $\pea^b \in L^{\infty}((0,T),H^1(\Oeaf))$ fulfilling the same estimates as in \ref{ass:Stokes_f} and \ref{ass:Stokes_pb}, with additional $L^{\infty}$-regularity with respect to time. This leads to the same a priori estimates for $\uea$ and $\pea$ as in Section \ref{sec:apriori_fluid}, with additional $L^{\infty}$-regularity in time. Further, the compactness results from Section \ref{sec:compactness_fluid} remain valid, where the limit function are $L^{\infty}$ with respect to time.  Further, we assume that $n \le 4$, to guarantee the existence of a weak microscopic solution.
For the diffusion coefficient $\Dea$ we consider the following cases with $D>0$ fixed (already given in Section \ref{sec:MainResults}):
\begin{enumerate}[label = (D\arabic*)]
    \item $\Dea =  \eps^{\alpha}D I \in \R^{n \times n}$, 

    \item $\Dea = D \mathrm{diag}(\eps^{-\alpha},\ldots , \eps^{-\alpha},  \eps^{\alpha} ) \in \R^{n\times n}$.
\end{enumerate}

Let us give the definition of a weak solution in the case \ref{case:diffusion_low}: We say that $\cea $ is a weak solution of the problem $\eqref{eq:problem_transport_micro}$ (for diffusion coefficient given in \ref{case:diffusion_low}) if $\cea \in L^2((0,T),H^1(\Oeaf)) $ with $\partial_t \cea \in L^2((0,T),H^1(\Oeaf, \Sea^+ \cup \Sea^-)')$  such that $\cea = \ceps^b$ on $\Sea^+ \cup \Sea^-$ and for all $\psiea \in H^1(\Oeaf)$ with $\psiea = 0 $ on $\Sea^+ \cup \Sea^-$  it holds almost everywhere in $(0,T)$ that
\begin{align}\label{eq:weak_transport_vea}
\frac{1}{\eps^{\alpha}}\langle \partial_t \cea , \psiea \rangle_{H^1(\Oeaf)} + \int_{\Oeaf}  \Dea\nabla \cea \cdot \nabla \psiea - \frac{\uea}{\eps^2} \cea \nabla \psiea \dd x = \frac{1}{\eps^{\alpha}}\int_{\Oeaf} g_{\eps,\alpha} \psiea \dd x, 
\end{align}
together with the initial condition $\cea(0) = 0$.
Introducing the quantity 
\begin{align*}
    \wea  := \cea - \ceps^b
\end{align*}
we obtain $\wea = 0$ on $\Sea^+ \cup \Sea^-$  and this function fulfills 
\begin{align}
\begin{aligned}\label{eq:weak_transport_wea}
    \frac{1}{\eps^{\alpha}}\langle \partial_t \wea , \psiea \rangle_{H^1(\Oeaf)}& + \int_{\Oeaf}  \Dea \nabla \wea \cdot \nabla \psiea - \frac{\uea}{\eps^2} \wea \nabla \psiea \dd x 
    \\
    &= \frac{1}{\eps^{\alpha}}\int_{\Oeaf} \left(g_{\eps,\alpha} -  \partial_t \ceps^b\right) \psiea \dd x + \int_{\Oeaf} \left(\Dea\nabla  \ceps^b - \frac{\uea}{\eps^2} \ceps^b \right) \cdot \nabla \psiea \dd x.
\end{aligned}
\end{align}
In the case of high diffusion in the horizontal direction \ref{case:diffusion_high}, we have to consider in the definition above $\cea \in L^2((0,T),H_{\#}^1(\Oeaf))$ and $\partial_t \cea \in L^2((0,T),H_{\#}^1(\Oeaf,\Sea^+ \cup \Sea^-)')$ and test-functions $\psiea \in H_{\#}^1(\Oeaf)$ (recall that $\#$ indicates the $\Sigma$-periodicity).

\noindent\textbf{Assumptions on the data:}
\begin{enumerate}[label = (T\arabic*)]
    \item The source term $g_{\eps,\alpha} \in L^{\infty}((0,T)\times \Oeaf)$ fulfills
\begin{align*}
\|g_{\eps,\alpha}\|_{L^{\infty}((0,T)\times \Oeaf)} \le C.
\end{align*}
Further, there exists $g_0 \in L^2((0,T)\times \Omega \times Y_f)$ such that $\chi_{\Oeaf} g_{\eps,\alpha} \toa g_0$.

\item 
The boundary-value $\ceps^b$ fulfills 
\begin{align*}
    \ceps^b \in L^2((0,T),H^1(\Oeaf))\cap H^1((0,T),L^2(\Oeaf)),
\end{align*}
with $\ceps^b(0) = 0$, such that 
\begin{align*}
    \|\ceps^b\|_{L^{\infty}((0,T) \times \Oeaf)}  + \|\partial_t \ceps^b\|_{L^2((0,T)\times \Oeaf)} \le C\eps^{\frac{\alpha}{2}}.
\end{align*}
In the case \ref{case:diffusion_high}, we additionally assume $\cea^b \in L^2((0,T),H_{\#}^1(\Oeaf))$.
For the bound of the gradient we consider two different cases:
\begin{itemize}
\item For $\Dea = \eps^{\alpha}D$ we assume: 
\begin{align*}
    \eps^{\alpha} \|\nabla \ceps^b\|_{L^{2}((0,T)\times \Oea)} \le C\eps^{\frac{\alpha}{2}}
\end{align*}
Further, there exists  (see also Proposition \ref{prop:compactness_v0_dn_v0}) $c_0^b \in L^2((0,T)\times \Omega)$ with $\partial_n v_0 \in L^2((0,T)\times \Omega)$ and $c_1^b \in L^2(\Omega,H^1_{\per}(Y)/\R)$, such that
\begin{align*}
   \chi_{\Oeaf} \ceps^b \toa  \chi_{Y_f}c_0^b, \qquad   \chi_{\Oeaf}\eps^{\alpha} \nabla \ceps^b \toa \chi_{Y_f}\left( \partial_n c_0^b e_n + \nabla_y c_1^b\right).
\end{align*}

\item For $\Dea = D \mathrm{diag}(\eps^{-\alpha},\ldots , \eps^{-\alpha},  \eps^{\alpha} ) \in \R^{n\times n}$ we assume:
\begin{align*}
 \|\nabla_{\ox} \ceps^b\|_{L^{2}((0,T)\times \Oea)} 
 + \eps^{\alpha} \|\partial_n \ceps^b\|_{L^{\infty}((0,T)\times \Oea)} \le C \eps^{\frac{\alpha}{2}}
\end{align*}
Further, there exists (see also Proposition \ref{prop:compactness_v0_nabla_v0}) $c_0^b \in L^2((0,T), H^1( \Omega))$  and $c_1^b \in L^2(\Omega,H^1_{\per}(Y)/\R)$ such that
\begin{align*}
      \chi_{\Oeaf} \ceps^b \toa \chi_{Y_f} c_0^b, \qquad    \chi_{\Oeaf}(\nabla_{\ox} \ceps^b , \eps^{\alpha} \partial_n \ceps^b) \toa \chi_{Y_f}\left( \nabla c_0^b + \nabla_y c_1^b \right).
\end{align*}
\end{itemize}
 \item\label{ass:shift_f} It holds that
    \begin{align*}
        \|\delta & f_{\eps,\alpha} \|_{L^2((0,T)\times\Oeah^f)} +  \|\nabla \delta \pea^b \|_{L^2((0,T)\times\Oeah^f)}  +   \|\delta g_{\eps,\alpha}\|_{L^2((0,T)\times\Oeah^f)}  
        \\
        &+ \|\partial_t\delta \ceps^b \|_{L^2((0,T)\times\Oeah^f)} +  \| \delta \ceps^b \|_{L^2((0,T)\times\Oeah^f)}  +   \eps^{\alpha}\| \nabla \delta g_{\eps,\alpha}\|_{L^2((0,T)\times\Oeah^f)}\le \kappa(|l\eps|) \eps^{\frac{\alpha}{2}}.
    \end{align*}
\end{enumerate}

\begin{remark}\ 
\begin{enumerate}[label = (\roman*)]

\item Our focus is the treatment of the advective term and the different scalings for the diffusion coefficient. Therefore we have chosen a homogeneous initial condition and  a linear reaction term. However, it is straightforward to extend our results to more general data.

\item Of course, due to our assumptions we can expect more regularity for the time-derivative. However, we show homogenization and dimension reduction (in particular the strong two-scale compactness results) for the time-derivative being a functional in the dual space of $H^1(\Oeaf,\Sea^+\cup \Sea^-)$ (respectively, with $\Sigma$-periodic boundary conditions), to provide methods for more general data.

\end{enumerate}

\end{remark}

\begin{corollary}
There exists a unique weak solution of the microscopic problem $\eqref{eq:problem_transport_micro}$.
\end{corollary}
\begin{proof}
This result is standard for fixed $\eps$ and can be obtained by the Galerkin method.
\end{proof}

\subsection{A priori estimates for the microscopic solution $\cea$}
\label{sec:apriori_wea}

We derive uniform a priori estimates with respect to  $\eps$ and $\alpha$. Of course, we will obtain different estimates for the gradient depending on the choice of $\Dea$ for the cases \ref{case:diffusion_low} and \ref{case:diffusion_high}. However, the ideas are the same for both cases and we can follow a standard procedure for energy estimates to obtain $L^2$-bounds for $\cea$ and its gradient. A more critical part is to obtain a uniform bound for the time-derivative. First of all, we need $L^{\infty}$-bounds for the concentration $\cea$ to control the advective term uniformly in $\eps$. Next, to establish later strong (two-scale) convergence for $\cea$ via a Kolmogorov-Simon type compactness argument, we need bounds for the time-derivative in dual spaces of Sobolev functions with weighted (with respect to $\eps$ and $\alpha$) norms adapted to the a priori bounds for $\cea$ and $\nabla \cea$ in $L^2$. Finally, for the strong convergence of $\cea$, an additional control with respect to the spatial variable is necessary, and therefore we give an additional estimate for the differences of shifts of the microscopic solutions $\cea $ and $\uea$. 
\\

\noindent\textbf{The case $\Dea = \eps^{\alpha} D$:}
\\

We test equation $\eqref{eq:weak_transport_wea}$ with $\psiea = \wea$ and obtain
\begin{align*}
\frac{1}{2\eps^{\alpha}} \frac{d}{dt}\|\wea\|_{L^2(\Oeaf)}^2& +   \eps^{\alpha}D \|\nabla \wea \|_{L^2(\Oeaf)}^2  - \frac12 \int_{\Oeaf} \frac{\uea}{\eps^2}  \nabla \wea^2 \dd x 
    \\
    &= \frac{1}{\eps^{\alpha}} \int_{\Oeaf} \left(g_{\eps,\alpha} - \partial_t \ceps^b\right) \wea \dd x + \int_{\Oeaf} \left( \eps^{\alpha} D\nabla  \ceps^b - \frac{\uea}{\eps^2} \ceps^b \right) \cdot \nabla \wea \dd x.
\end{align*}
For the convective term we can use integration by parts together with the zero boundary conditions of $\uea$ and $\wea$, and also the divergence-free condition of $\uea$ to obtain
\begin{align*}
    \frac12 \int_{\Oeaf} \frac{\uea}{\eps^2}  \nabla \wea^2 \dd x  = 0.
\end{align*}
It remains to estimate the terms on the right-hand side. For the first term, we obtain with the assumptions on $g_{\eps,\alpha}$ and $\partial_t \ceps^b$
\begin{align*}
     \frac{1}{\eps^{\alpha}} \int_{\Oeaf} \left(g_{\eps,\alpha} - \partial_t \ceps^b\right) \wea \dd x \le \frac{C}{\eps^{\frac{\alpha}{2}}} \|\wea\|_{L^2(\Oeaf)} \le C  \left(1 + \frac{1}{\eps^{\alpha}} \|\wea \|_{L^2(\Oeaf)}^2 \right).
\end{align*}
For the second term on the right-hand side in the above equation  we consider separately the diffusive and advective term. For the first one, we have
\begin{align}
\begin{aligned}\label{ineq:apriori_aux}
\int_{\Oeaf} \eps^{\alpha} D\nabla  \ceps^b  \cdot \nabla \wea \dd x   \le C \eps^{\frac{\alpha}{2}} \|\nabla \wea\|_{L^2(\Oeaf)} 
\le C(\theta)    + \theta \eps^{\alpha} \|\nabla\wea\|_{L^2(\Oeaf)}^2
\end{aligned}
\end{align}
for arbitrary $\theta >0$.
For the  advective term  we use the essential boundedness of $\ceps^b$ and the a priori bound for $\uea $ from Proposition \ref{prop:apriori_fluid_velocity}, to obtain
\begin{align*}
 \int_{\Oeaf}  \frac{\uea}{\eps^2} \ceps^b \cdot \nabla \wea \dd x   \le  C \left\|\frac{\uea}{\eps^2} \right\|_{L^2(\Oeaf)} \|\nabla \wea \|_{L^2(\Oeaf)}
 \le  C(\theta) + \theta \eps^{\alpha} \|\nabla\wea\|_{L^2(\Oeaf)}^2.
\end{align*}
Altogether, choosing $\theta$ small enough, we can use an absorption argument, then we integrate with respect to time and use the Gronwall-inequality to obtain
\begin{align*}
    \frac{1}{\eps^{\frac{\alpha}{2}}}\| \wea\|_{L^{\infty}((0,T),L^2( \Oeaf))} + \eps^{\frac{\alpha}{2}} \|\nabla \wea\|_{L^2((0,T)\times \Oeaf)} \le C .
\end{align*}
Due to the assumptions on $\ceps^b$, we obtain the same estimate also for $\cea$.
We summarize our results in the following Proposition.

\begin{proposition}\label{prop:apriori_concentration}
For $\Dea = \eps^{\alpha} D$  it holds that 
\begin{align*}
     \| \wea\|_{L^{\infty}((0,T),L^2( \Oeaf))} + \eps^{\alpha} \|\nabla \wea\|_{L^2((0,T)\times \Oeaf)} \le C \eps^{\frac{\alpha}{2}}.
\end{align*}
The same estimate is valid for $\cea$ instead of $\wea$.
\end{proposition}

\noindent\textbf{The case $\Dea = D \mathrm{diag}(\eps^{-\alpha},\ldots , \eps^{-\alpha},  \eps^{\alpha} ) \in \R^{n\times n}$:}
\\

We proceed in the same way as in the previous case, testing $\eqref{eq:problem_transport_micro}$ with $\wea$ to obtain
\begin{align*}
\frac{1}{2\eps^{\alpha}} \frac{d}{dt}\|\wea\|_{L^2(\Oeaf)}^2& +   \frac{1}{\eps^{\alpha}}D \|\nabla_{\ox} \wea \|_{L^2(\Oeaf)}^2  + \eps^{\alpha} D\|\partial_n \wea\|^2_{L^2(\Oeaf)} - \frac12 \int_{\Oeaf} \frac{\uea}{\eps^2}  \nabla \wea^2 \dd x 
    \\
    =& \frac{1}{\eps^{\alpha}} \int_{\Oeaf} \left(g_{\eps,\alpha} - \partial_t \ceps^b\right) \wea \dd x - \int_{\Oeaf}  \frac{\uea}{\eps^2} \ceps^b \cdot \nabla \wea \dd x
    \\
    &+ \int_{\Oeaf} \frac{1}{\eps^{\alpha}} D\nabla_{\ox} \ceps^b \cdot \nabla_{\ox} \wea  + \eps^{\alpha} D\partial_n \ceps^b \partial_n \wea \dd x.
\end{align*}
Now, the main difference from the proof of Proposition \ref{prop:apriori_concentration} lies in the estimate of the last term on the right-hand side. Here, we can  argue in the same way as for $\eqref{ineq:apriori_aux}$ separately for $\nabla_{\ox}$ and $\partial_n$.  In summary, we get
\begin{proposition}\label{prop:apriori_concentration_diff_nabla}
For $\Dea = D \mathrm{diag}(\eps^{-\alpha},\ldots , \eps^{-\alpha},  \eps^{\alpha} ) \in \R^{n\times n}$ it holds that
\begin{align*}
 \|\wea\|_{L^2(\Oea)} + \|\nabla_{\ox} \wea\|_{L^2(\Oea)} +  \eps^{\alpha} \|\partial_n \wea \|_{L^2(\Oea)} \le C \eps^{\frac{\alpha}{2}}.
\end{align*}
The same estimate is valid for $\cea$ instead of $\wea$.
\end{proposition}

\noindent\textbf{Estimates for the time-derivative $\partial_t \wea$:}
To control the time-derivative, we need control for the convective term. As usual when dealing with flow in porous medium, the embedding $H^1$ into $L^4$ (at least for $n\leq 4$) is not applicable, since the gradient of $\uea$ scales badly with respect to $\eps$. To overcome this problem, it is a standard approach to show $L^{\infty}$-bounds for the concentration $\cea$. The proof follows the same lines as in the case of full (perforated domains), so we only give a brief sketch. 

\begin{lemma}\label{lem:apriori_Linfty}
For both cases \ref{case:diffusion_low} and \ref{case:diffusion_high} for the diffusion coefficient $\Dea$, it holds that
\begin{align*}
    \|\cea \|_{L^{\infty}((0,T)\times \Oeaf)} \le C.
\end{align*}
\end{lemma}
\begin{proof}
We only sketch the main ideas and refer for example to the proof of \cite[Lemma 5.2]{gahn_wiedemann_etal_2025} for more details. We emphasize that now we work directly with the weak formulation of $\cea$ instead of $\wea$. We define $W:= e^{-\omega t} \cea $ with $\omega>0$ defined below and $t \in (0,T)$. Further, we put $W_k:= W - k $ for $k\in \N$ and $W_k^+:= (W-k)^+$ with $(\cdot)^+ := \max\{0,\cdot\}$, and use $e^{-\omega t} W_k^+$ as a test-function in $\eqref{eq:weak_transport_vea}$. This is an admissible test-function for $k> \|\ceps^b\|_{L^{\infty}((0,T)\times \Oea)}$, since then $W_k^+ = 0 $ on $\Sea^{\pm}$. We  obtain after integration in time from $0$ to $t \in [0,T]$
\begin{align*}
    \frac{1}{2\eps^{\alpha}} \|W_k^+(t)\|_{L^2(\Oeaf)}^2 + \int_0^t &\int_{\Oeaf} \Dea \nabla W_k^+ \cdot \nabla W_k^+ \dd x\dd s 
    \\
    =& \int_0^t \int_{\Oeaf} \frac{\uea}{\eps^2} W \cdot \nabla W_k^+ \dd x\dd s + \frac{1}{\eps^{\alpha}}\int_0^t \int_{\Oeaf} e^{-\omega s} g_{\eps,\alpha} W_k^+ \dd x\dd s.
\end{align*}
For the convective term we use integration by parts to get
\begin{align*}
\int_{\Oeaf} \frac{\uea}{\eps^2} W \cdot \nabla W_k^+ \dd x  = \frac12 \int_{\Oeaf} \frac{\uea}{\eps^2} \nabla (W_k^+)^2 \dd x + k \int_{\Oeaf} \frac{\uea}{\eps^2} \nabla W_k^+ \dd x = 0,
\end{align*}
where in the last equality we used the zero boundary condition of $\uea$ and $W_k^+$. The force term can be estimated in the following way by using the $L^{\infty}$ bound for $g_{\eps,\alpha}$
\begin{align*}
\frac{1}{\eps^{\alpha}}\int_0^t \int_{\Oeaf} e^{-\omega s} g_{\eps,\alpha} W_k^+ \dd x\dd s \le \frac{C}{\eps^{\alpha}} \left( \int_0^t \int_{\{W_k>k\}} \dd x\dd s +  \|W_k^+\|_{L^2((0,t)\times \Oeaf)}^2 \right) .
\end{align*}
Now, the Gronwall-inequality and \cite[II Theorem 6.1 and Remark 6.2]{Ladyzenskaja} imply the desired result.
\end{proof}

Now, we are able to estimate the time-derivative $\partial_t \cea$. Here, it is necessary to estimate the norm in the dual space of functions spaces suitably scaled with respect to $\eps$ and $\alpha$, and therefore in particular depending on the choice of $\Dea$. For this we introduce the space
$\spaceH_{\Dea}$ consisting of functions in $H^1(\Oeaf, \Sea^+ \cup \Sea^-) $ in the case \ref{case:diffusion_low} and $H^1_{\#}(\Oeaf, \Sea^+ \cup \Sea^-)$ in the case \ref{case:diffusion_high},  together with the norm
\begin{align*}
   \|\psiea\|_{\spaceH_{\Dea}}^2 := \frac{1}{\eps^{\alpha}}\|\psiea\|_{L^2(\Oeaf)}^2 + \|\sqrt{\Dea} \nabla \psiea\|_{L^2(\Oeaf)}^2.  
\end{align*}
We emphasize, that from our a priori estimates on $\cea$ above we have 
\begin{align*}
    \|\cea\|_{L^2((0,T),\spaceH_{\Dea})} \le C.
\end{align*}

\begin{proposition}\label{prop:apriori_time_derivative}
It holds that
\begin{align*}
    \frac{1}{\eps^{\alpha}} \|\partial_t \cea\|_{L^2((0,T),\spaceH_{\Dea}')} \le C.
\end{align*}
\end{proposition}
\begin{proof}
We test equation $\eqref{eq:weak_transport_vea}$ with $\psiea \in \spaceH_{\Dea}$ such that $\|\psiea\|_{\spaceH_{\Dea}} \le 1$. In particular, we have
\begin{align*}
    \|\nabla \psiea \|_{L^2(\Oeaf)}  \le C\eps^{-\frac{\alpha}{2}}.
\end{align*}

We get almost everywhere in $(0,T)$
\begin{align*}
    \frac{1}{\eps^{\alpha}}\bigg|&\langle \partial_t \cea , \psiea \rangle_{H^1(\Oeaf)}\bigg| 
    \\
    =& \left| - \int_{\Oeaf}  \Dea\nabla \cea \cdot \nabla \psiea - \frac{\uea}{\eps^2} \cea \nabla \psiea \dd x + \frac{1}{\eps^{\alpha}}\int_{\Oeaf} g_{\eps,\alpha} \psiea \dd x \right|
    \\
    \le& C \|\sqrt{\Dea} \nabla \cea \|_{L^2(\Oeaf)} \|\sqrt{\Dea}\nabla \psiea\|_{L^2(\Oeaf)} 
    \\
    & + \left\|\frac{\uea}{\eps^2}\right\|_{L^2(\Oeaf)} \|\cea\|_{L^{\infty}((0,T)\times \Oeaf)} \|\nabla \psiea \|_{L^2(\Oeaf)} + \eps^{-\alpha} \|g_{\eps,\alpha} \|_{L^2(\Oeaf)} \|\psiea\|_{L^2(\Oeaf)}
    \\
    \le& C \|\sqrt{\Dea} \nabla \cea \|_{L^2(\Oeaf)}  + C \eps^{-\frac{\alpha}{2}}\left\|\frac{\uea}{\eps^2}\right\|_{L^2(\Oeaf)} + C ,
\end{align*}
where we used the $L^{\infty}$ bound for $g_{\eps,\alpha}$ and $\cea$ obtained in Lemma \ref{lem:apriori_Linfty}. Taking the supremum over $\psiea$, squaring, integrating with respect to time and using the a priori estimates for $\cea$ and $\uea$, we get the desired result. 
\end{proof}

\noindent\textbf{Estimates for the shifts:}
To obtain strong convergence (in the two-scale sense),
more control on the spatial variable is necessary. For this, we introduce the following notation for the differences of shifted functions. Let $\psiea : \R^{n-1} \times (-\eps^{\alpha},\eps^{\alpha}) \rightarrow \R$ and $l \in \Z^{n-1} \times \{0\}$. We define  
\begin{align*}
\delta \psiea := \psiea(\cdot + l\eps)  - \psiea,
\end{align*}
where in this notation we neglect the dependence on $l$ and $\eps$, which should be clear from the context. In the following, we extend the function $\uea$ by zero to $\R^{n-1} \times (-\eps^{\alpha},\eps^{\alpha})$ and the function $\cea$ first with the extension operator from Lemma \ref{lem:extension_operator} below to $\Oea$, and then in an arbitrary smooth way to $\R^{n-1}\times (-\eps^{\alpha},\eps^{\alpha})$, such that the a priori estimates,  in particular the $L^{\infty}$-estimate, remain valid (this can be done by mirroring). We use the same notation for both extensions as before.

\begin{lemma}[Extension operator]\label{lem:extension_operator}
There exists an extension operator $E_{\eps}:H^1(\Oeaf) \rightarrow H^1(\Oea)$ such that for all $\vea \in H^1(\Oeaf)$ it holds that
\begin{align*}
    \|E_{\eps} \vea\|_{L^2(\Oea)} \le C \|\vea\|_{L^2(\Oeaf)} , \qquad \|\nabla E_{\eps} \vea\|_{L^2(\Oea)} \le C \|\nabla \vea \|_{L^2(\Oeaf)}.
\end{align*}
If additionally  $\vea \in L^{\infty}(\Oeaf)$, it holds that
\begin{align*}
    \|E_{\eps}\vea\|_{L^{\infty}(\Oea)} \le C \|\vea \|_{L^{\infty}(\Oeaf)}.
\end{align*}
\end{lemma}
\begin{proof}
    This result can be shown as in \cite{Acerbi1992}. The fact that that here we deal with a thin layer with thickness of order $\eps^{\alpha}$ has no influence. For the inequality for the $L^{\infty}$-bound, we refer to \cite[Lemma A.3]{bhattacharya2022homogenization}.
\end{proof}

Next, we construct a domain $\Oeah^f$ for $h$ small, which is obtained by cutting of micro-cells from $\Oeaf$ near to the lateral boundary (distance smaller than $h$). More precisely, we introduce the following notation: For $0 < h \ll 1$ let  $\Sigma_h := \left\{x \in \Sigma\, : \, \mathrm{dist}\{x,\partial \Sigma\} > h\right\}$ and we set
\begin{align*}
    K_{\eps,h}:= \left\{ k \in \Z^{n-1} \, : \, \eps (k + (0,1)^{n-1}) \subset \Sigma_h \right\}
\end{align*}
and define
\begin{align*}
\Sigma_{\eps,h} := \mathrm{int}\left\{\bigcup_{k \in K_{\eps,h}} \eps( [0,1]^{n-1} + k)\right\}.
\end{align*}
In other words, $\Sigma_{\eps,h}$ consists of all points in $x$ with distance greater than $h$ and included in a microscopic cell $\eps(k + [0,1]^{n-1})$ strictly contained in $\Sigma_h$. Now, we define 
\begin{align*}
    \Oeah:= \Sigma_{\eps,h} \times (-\eps^{\alpha},\eps^{\alpha}), \qquad \Oeah^f:= \Oeah \cap \Oeaf.
\end{align*}

\begin{proposition}\label{prop:estimate_shifts}
We obtain for every $0 < h \ll 1$ a constant $C_h>0$ depending on $h$ (but independent of $\eps$), such that for every $l \in \Z^{n-1}\times \{0\}$ and $|l\eps|<h$ it holds that
\begin{align*}
\eps^{-\frac{\alpha}{2}} \| \delta \cea\|_{L^{\infty}((0,T),L^2(\Oea^f)} \le  C\sqrt{h} + C_h\eps^{\frac{\alpha}{2}} + \kappa(|l\eps|),
\end{align*}
with $\kappa(s) \rightarrow 0$ for $s\to 0$. In the case \ref{case:diffusion_high} with periodic boundary conditions (after extending $\cea$ periodically in $\Sigma$-direction), the inequality is valid for $h=0$ and arbitrary $l$ and $\eps$, and the constant on the right-hand side is independent of $h$. 

\end{proposition}
\begin{proof}
We first consider the case \ref{case:diffusion_low} with $\Dea = \eps^{\alpha}D$. We use similar ideas as in the proof of \cite[Lemma 4.3]{GahnNeussRaduKnabner2018a}, where here we have to estimate additionally the convective term. We define the space
\begin{align*}
    \spaceH_{\eps,h}:= \left\{ \phi \in H^1(\Oeah^f) \, : \, \phi = 0 \mbox{ on } \partial \Oeah^f \setminus \Gamma_{\eps,\alpha}  \right\}.
\end{align*}

Let $l\in \Z^{n-1}\times \{0\}$, such that $|l\eps|<h$.
It is easy to check, that for all $\psiea \in \spaceH_{\eps,h}$ it holds almost everywhere in $(0,T)$ that 
\begin{align}
\begin{aligned}\label{eq:aux_estimate_shifts_basic}
\frac{1}{\eps^{\alpha}} & \langle \partial_t \delta \cea , \psiea\rangle_{\spaceH_{\eps,h}} + \int_{\Oeah^f} \eps^{\alpha} D \nabla \delta \cea \cdot \nabla \psiea \dd x
\\
&- \int_{\Oeah^f} \frac{\delta(\uea \cea)}{\eps^2} \cdot \nabla \psiea \dd x
= \frac{1}{\eps^{\alpha}} \int_{\Oeah^f} \delta g_{\eps,\alpha} \psiea \dd x.
\end{aligned}
\end{align}
First, we assume that $\ceps^b = 0$.
We choose a cut-off function $\eta \in C_0^{\infty}(\Sigma_h)$ with $0\le \eta \le 1$  and $\eta = 1$ in $\Sigma_{2h}$. We emphasize that $\eta$ is depending on $h$ and in particular the gradient is of order $\frac{1}{h}$. In the following, we denote by $C_h$ constants which depend on $h$ (and might grow to $\infty$ for $h \to 0$). Now, we choose $\psiea = \eta^2 \delta \cea  \in \spaceH_{\eps,\alpha}$ and we have
\begin{align*}
    \nabla \psiea = \eta \nabla (\eta \delta \cea) + \eta \delta \cea \nabla \eta = 2 \eta \delta \cea  \nabla \eta + \eta^2 \nabla \delta \cea. 
\end{align*}
We get for all $t \in [0,T]$
\begin{align}
\begin{aligned}\label{eq:aux_estimate_shifts}
\frac{1}{2\eps^{\alpha}}&  \|\eta \delta \cea (t) \|^2_{L^2(\Oeah^f)} + \eps^{\alpha} D\|\eta \nabla \delta \cea \|_{L^2((0,t)\times \Oeah^f)}^2 + 2 \eps^{\alpha} D \int_0^t\int_{\Oeah^f}  \nabla \delta \cea \cdot \nabla \eta \eta \delta \cea  \dd x\dd s
\\
&- \int_0^t\int_{\Oeah^f} \frac{\delta(\uea \cea)}{\eps^2} \cdot \nabla (\eta^2 \delta \cea ) \dd x\dd s = \frac{1}{\eps^{\alpha}} \int_0^t \int_{\Oeah^f} \delta g_{\eps,\alpha} \eta^2 \delta\cea \dd x\dd s.
\end{aligned}
\end{align}
For the third term on the left-hand side we get
\begin{align*}
 \eps^{\alpha}\int_0^t   \int_{\Oeah^f}  \nabla \delta \cea \cdot \nabla \eta \eta \delta \cea  \dd x\dd s &\le \frac{1}{\eps^{\alpha}} \|\eta \delta \cea\|_{L^2((0,t)\times \Oeah^f)}^2 + C_h \eps^{3\alpha} \| \nabla \delta \cea\|_{L^2((0,t)\times \Oeah^f)}^2 
 \\
 &\le \frac{1}{\eps^{\alpha}} \|\eta \delta \cea\|_{L^2((0,t)\times \Oeah^f)}^2 + C_h \eps^{2\alpha} ,
\end{align*}
where in the last inequality we used the a priori estimate for $\nabla \cea$ from Proposition \ref{prop:apriori_concentration}.
For the convective term, we use 
\begin{align*}
\int_{\Oeah^f} \frac{\uea}{\eps^2} \nabla (\eta \delta \cea )^2 \dd x = 0,
\end{align*}
to obtain with $\delta(\uea \cea ) = \cea \delta \uea + \uea(\cdot + l\eps ) \delta \cea$  and the a priori estimates from Proposition \ref{prop:apriori_fluid_velocity}, \ref{prop:apriori_concentration} and \ref{prop:apriori_concentration_diff_nabla}
\begin{align*}
\bigg| \int_0^t &\int_{\Oeah^f} \frac{\delta(\uea \cea)}{\eps^2} \cdot \nabla \psiea \dd x\dd s  \bigg|
\\
\le& \left| \int_0^t \int_{\Oeah^f} \frac{\delta \uea}{\eps^2} \cea \cdot \left[ 2 \eta \delta \cea \nabla \eta + \eta^2 \nabla \delta\cea \right] + \frac{\uea}{\eps^2} \delta \cea \left[ \eta \nabla (\eta \delta \cea) + \eta \delta \cea \nabla \eta \right] \dd x\dd s \right|
\\
\le& C_h \left\|\frac{\delta \uea }{\eps^2}\right\|_{L^2((0,t)\times \Oeah^f)} \|\cea \|_{L^{\infty}((0,T)\times \Oeah^f)} \|\eta \delta \cea \|_{L^2((0,t)\times \Oeah^f)} 
\\
&+ C \left\|\frac{\eta \delta \uea }{\eps^2}\right\|_{L^2((0,t)\times \Oeah^f)} \|\cea \|_{L^{\infty}((0,T)\times \Oeah^f)} \|\eta \nabla \delta \cea \|_{L^2((0,t)\times \Oeah^f)}
\\
&+ C_h \left\|\frac{ \uea }{\eps^2}\right\|_{L^2((0,t)\times \Oeah^f)} \|\delta\cea \|_{L^{\infty}((0,T)\times \Oeah^f)} \|\eta \delta \cea \|_{L^2((0,t)\times \Oeah^f)} 
\\
\le& C_h \eps^{\frac{\alpha}{2}} \|\eta \delta \cea \|_{L^2((0,t)\times \Oeah^f)}  + C \left\| \frac{\eta \delta \uea }{\eps^2} \right\|_{L^2((0,t)\times \Oeah^f)}\|\eta \nabla \delta \cea \|_{L^2((0,t)\times \Oeah^f)} \\
&+ C_h \eps^{\frac{\alpha}{2}} \|\eta \delta \cea \|_{L^2((0,t)\times \Oeah^f)}
\\
\le& C_h \eps^{\alpha} + \frac{1}{\eps^{\alpha}} \|\eta \delta \cea \|_{L^2((0,t)\times \Oeah^f)}^2 + \frac{C}{\eps^{\alpha}} \left\| \frac{\eta \delta \uea}{\eps^2}\right\|_{L^2((0,t)\times \Oeah^f)}^2 + \frac{D \eps^{\alpha}}{2} \|\eta \nabla \delta \cea \|_{L^2((0,t)\times \Oeah^f)}^2.
\end{align*}
The term on the right-hand side of $\eqref{eq:aux_estimate_shifts}$ can be estimated in a similar way. In total we get with an absorption argument and the Gronwall inequality
\begin{align}
\begin{aligned}\label{ineq:aux_difference_shifts}
    \frac{1}{\eps^{\frac{\alpha}{2}}} &\|\eta \delta \cea\|_{L^{\infty}((0,T),L^2(\Oeah^f)} + \eps^{\frac{\alpha}{2}} \|\eta \nabla \delta \cea \|_{L^2((0,T)\times \Oeah^f)}
    \\
    &\le C_h \eps^{\frac{\alpha}{2}} + \frac{C}{\eps^{\frac{\alpha}{2}}} \left\|\frac{\eta \delta \uea }{\eps^2} \right\|_{L^2((0,T)\times \Oeah^f)} + \frac{C}{\eps^{\frac{\alpha}{2}}} \|\delta g_{\eps,\alpha}\|_{L^2((0,T)\times \Oeah^f)}  .
\end{aligned}
\end{align}
It remains to estimate the term including $\delta \uea$. For this, we consider the equation for $\delta \uea$. More precisely, for all $\phiea \in H^1(\Oeah^f)^n$, such that $\phiea = 0$ on $\partial \Oeah^f \setminus (\Sea^+ \cup \Sea^-)$, it holds that 
\begin{align*}
\int_{\Oeah^f} \nabla \delta \uea : \nabla \phiea \dd x - \int_{\Oeah^f} (\delta \pea - \delta\pea^b) \nabla \cdot \phiea \dd x = \int_{\Oeah^f} (\delta f_{\eps,\alpha} - \nabla \delta \pea^b) \cdot \phiea \dd x.
\end{align*}
Now, we choose $\phiea = \eta^2 \delta \uea $ to obtain with $\nabla \phiea = \eta^2 \nabla \delta \uea + 2\eta \nabla \eta \otimes \delta \uea $ and $\nabla \cdot \phiea = 2\eta \nabla \eta \cdot \delta \uea$
\begin{align*}
\|\eta \nabla \delta \uea \|_{L^2(\Oeah^f)}^2 =& \int_{\Oeah^f} (\delta \pea - \delta \pea^b) 2\eta \nabla \eta \cdot \delta \uea \dd x + \int_{\Oeah^f}( \delta f_{\eps,\alpha} - \nabla \delta \pea^b) \cdot \delta \uea \eta^2 \dd x 
\\
&- 2\int_{\Oeah^f} \eta \nabla \delta \uea : (\nabla \eta \otimes \delta \uea ) \dd x 
\\
\le& C_h \|\delta \pea - \delta \pea^b\|_{L^2(\Oeah^f)} \|\delta\uea \|_{L^2(\Oeah^f)} 
\\
&+ C \| \delta f_{\eps,\alpha} - \nabla \delta \pea^b\|_{L^2(\Oeah^f)}  \|\eta \delta \uea \|_{L^2(\Oeah^f)}\\
&+ C_h \|\eta \nabla \delta \uea \|_{L^2(\Oeah^f)} \|\delta \uea \|_{L^2(\Oeah^f)}
\\
\le& C_h \eps^{2 + 2\alpha} + C \eps^{2 + \alpha } \kappa(|l\eps|) + C_h \eps^{2 + \frac{\alpha}{2}} \|\eta \nabla \delta \uea \|_{L^2(\Oeah^f)},
\end{align*}
where at the end we used the a priori estimates for $\uea$ and $\pea$ from Proposition \ref{prop:apriori_fluid_velocity} and \ref{prop:apriori_fluid_pressure}, and the  assumption \ref{ass:shift_f}. For the last term we can use the Young inequality to obtain with an absorption argument
\begin{align*}
\|\eta \nabla \delta \uea \|_{L^2(\Oeah^f)}^2 \le C_h \eps^{2 + 2\alpha} + C \eps^{2 + \alpha } \kappa(|l\eps|) 
\end{align*}
Using the Poincar\'e inequality from Lemma \ref{lem:Poincare}, we get
\begin{align*}
\|\eta \delta \uea\|_{L^2(\Oeah^f)} &\le C \eps \left( \|\delta \uea \nabla \eta \|_{L^2(\Oeah^f)} + \|\eta \nabla \delta \uea \|_{L^2(\Oeah^f)}   \right) 
\\
&\le C_h \eps^{3 + \frac{\alpha}{2}} + C_h \eps^{2 + \alpha} + C \eps^{2 + \frac{\alpha}{2}} \kappa(|l\eps|)
\end{align*}
In summary, we obtain
\begin{align}
    \eps^{-2} \|\eta \delta \uea\|_{L^2(\Oeah^f)} + \eps^{-1} \|\eta \nabla \delta \uea \|_{L^2(\Oeah^f)} \le C_h \eps^{\alpha} + C \eps^{\frac{\alpha}{2}} \kappa (|l\eps|).
\end{align}
Using this estimate in inequality $\eqref{ineq:aux_difference_shifts}$, we get with the assumption \ref{ass:shift_f}
\begin{align*}
   \frac{1}{\eps^{\frac{\alpha}{2}}} &\|\eta \delta \cea\|_{L^{\infty}((0,T),L^2(\Oeah^f)} + \eps^{\frac{\alpha}{2}} \|\eta \nabla \delta \cea \|_{L^2((0,T)\times \Oeah^f)}    \le C_h\eps^{\frac{\alpha}{2}} + \kappa(|l\eps|) .
\end{align*}
Using the $L^\infty$-estimate for $\cea$ from Lemma \ref{lem:apriori_Linfty} (and the fact that we consider here a smooth extension of $\cea$ to the whole layer $\R^{n-1}\times (-\eps^{\alpha},\eps^{\alpha})$), we obtain
\begin{align*}
\eps^{-\frac{\alpha}{2}} \| \delta \cea\|_{L^{\infty}((0,T),L^2(\Oea^f))} &\le \eps^{-\frac{\alpha}{2}} \|\eta \delta \cea\|_{L^{\infty}((0,T),L^2(\Oeaf \setminus \Omega_{\eps,\alpha,2h}^f))} +C \sqrt{h}
\\
&\le  C\sqrt{h} + C_h\eps^{\frac{\alpha}{2}} + \kappa(|l\eps|).
\end{align*}
This is the desired result for $\ceps^b = 0$. For the general case, we consider in the previous calculations $\wea$ instead of $\cea$. Then, we have to consider  in $\eqref{eq:aux_estimate_shifts}$ the function $g_{\eps,\alpha} - \partial_t \ceps^b$ instead of $g_{\eps,\alpha}$, and we obtain on the right-hand side the additional term
\begin{align*}
 \int_{\Oeaf} \left(\Dea\nabla  \ceps^b - \frac{\uea}{\eps^2} \ceps^b \right) \cdot \nabla (\eta^2 \cea) \dd x.
\end{align*}
This term can be estimated in the same way as the respective terms in the above calculation with $\ceps^b$ instead of $\cea$ and using the assumption \ref{ass:shift_f}. This finishes the proof in the case \ref{case:diffusion_low}.

Now, we consider the case \ref{case:diffusion_high}. Here, we can choose directly $\psiea = \delta \cea $ in $\eqref{eq:aux_estimate_shifts_basic}$ and we can work with the full domain $\Oeaf$ instead of $\Oeah^f$ (formally, we can choose $h=0$), since we can extend all the function periodically in the horizontal direction. In particular, in $\eqref{eq:aux_estimate_shifts}$ we can put $\eta = 0$ and all terms including $\nabla \eta$ vanish. The remaining terms can be estimate as in the case \ref{case:diffusion_high}.
\end{proof}

\begin{remark}\label{rem:estimate_shifts}\
\begin{enumerate}[label = (\roman*)]
    \item  We also showed that 
   \begin{align*}
        \eps^{\frac{\alpha}{2}} \| \nabla \delta \cea \|_{L^2((0,T)\times \Omega_{\eps,\alpha,2h}^f)}    \le C_h\eps^{\frac{\alpha}{2}} + \kappa(|l\eps|)
   \end{align*}
and an estimate for $\delta \uea$ and $\nabla \delta \uea$. However, for the proof of the strong convergence of $\cea$ this estimate is not necessary and therefore we only formulated the result for $\delta \cea$.

    \item The proof of Proposition \ref{prop:estimate_shifts} simplifies for the case of periodic boundary conditions. However, this assumption seems to be necessary in the case \ref{case:diffusion_high}. Otherwise (assuming also a Neumann-boundary condition), we get in $\eqref{eq:aux_estimate_shifts}$ the critical term
    \begin{align*}
        2\eps^{-\alpha} D \int_0^t \int_{\Oeah^f} \nabla_{\ox} \delta \cea \cdot \nabla\eta \eta \delta \cea \dd x\dd s.
    \end{align*}
    Using the same estimate as in the proof above, we only get
    \begin{align*}
        \left|  2\eps^{-\alpha} D \int_0^t \int_{\Oeah^f} \nabla_{\ox} \delta \cea \cdot \nabla\eta \eta \delta \cea \dd x\dd s \right| \le \frac{1}{\eps^{\alpha}} \|\eta \delta \cea \|_{L^2((0,t)\times \Oeah^f)}^2 + C_h.
    \end{align*}
    As we will see later in the proof of Proposition \ref{prop:strong_ts_different_scaling}, this is not enough to guarantee the strong two-scale convergence of $\cea$.
\end{enumerate}
  
\end{remark}

\subsection{Two-scale compactness for the microscopic solutions $\cea$}
\label{sec:compactness_transport}

We formulate the (two-scale) compactness results for the microscopic solution $\cea$ for the different choices of $\Dea$. The weak convergence results are direct consequences of the general two-scale compactness results obtained in Section \ref{sec:two_scale_compactness} and the a priori estimates in Section \ref{sec:apriori_wea}. However, to deal with the nonlinear advective term, we also need strong two-scale compactness results. For this, we use the additional bound for the differences of the shifts with respect to the spatial variable.
\\

In the following, we extend the functions $\cea$ with the extension operator $E_{\eps}$ from Lemma \ref{lem:extension_operator} to the whole thin layer $\Oea$ and use the same notation $\cea$ for the extension.  We emphasize that the a priori estimates in the case \ref{case:diffusion_high} for the extended function do not preserve, since we only have
\begin{align*}
    \|\nabla_{\ox} E_{\eps} \cea\|_{L^2((0,T)\times \Oea)} \le C \|\nabla \cea\|_{L^2((0,T)\times \Oeaf)} \le C\eps^{-\frac{\alpha}{2}}.
\end{align*}
Hence, due to the (arbitrary) shape of the perforations $Y_s$, we  can only control the horizontal gradient $\nabla_{\ox} E_{\eps} \cea$ by the full gradient $\nabla \cea$, including, in particular, the $n$-th component of $\cea$, scaling badly with respect to $\eps$.
We start with the formulation of the weak compactness results for the microscopic solution:
\begin{proposition}\label{prop:weak_two_scale_compactness_micro_sol_conc}\ 
\begin{enumerate}[label = (\roman*)]
    \item For $\Dea = \eps^{\alpha} D$  we have:
    
    There exist $c_0 \in L^2((0,T)\times \Omega)$ with $\partial_n c_0 \in L^2((0,T)\times \Omega)$ and $c_1 \in L^2((0,T)\times \Omega,H_{\per}^1(Y)/\R)$ such that up to a subsequence
\begin{align*}
   \chi_{\Oeaf} \cea \toa \chi_{Y_f} c_0,\qquad  \chi_{\Oeaf}\eps^{\alpha} \nabla \cea \toa \chi_{Y_f} \left( \partial_n c_0 e_n + \nabla_y c_1 \right).
\end{align*}
Further, we have $c_0 = c_0^b $ on $S_1^{\pm}$ in the (generalized) trace sense.

    \item For $\Dea = D \mathrm{diag}(\eps^{-\alpha},\ldots , \eps^{-\alpha},  \eps^{\alpha} ) \in \R^{n\times n}$ we have:
    \\
    There exist $c_0 \in L^2((0,T),H^1(\Omega))$ and $c_1 \in L^2((0,T)\times \Omega, H_{\per}^1(0,1)/\R)$ (the microscopic variable is the $y_n$-component), and $\bar{c}_1 \in L^2((0,T)\times \Omega,H^1_{\per,\nabla_{\oy}}(Y_f) )$ such that up to a subsequence
    \begin{align*}
    \chi_{\Oeaf}\cea \toa \chi_{Y_f} c_0 , \qquad \chi_{\Oeaf}(\nabla_{\ox} \cea , \eps^{\alpha} \partial_n \cea )\toa \chi_{Y_f} \left(\nabla c_0 + (\nabla_{\oy}\bar{c}_1,\partial_{y_n} c_1)\right).
    \end{align*}
    Further, we have $c_0 = c_0^b $ on $S_1^{\pm}$ in the  trace sense.
\end{enumerate}
\end{proposition}
\begin{proof}
The convergence results are a direct consequence of the a priori estimates from Proposition \ref{prop:apriori_concentration} and \ref{prop:apriori_concentration_diff_nabla}, and the compactness results from Proposition \ref{prop:compactness_v0_dn_v0} and \ref{prop:compactness_v0_nabla_v0}. 
It remains to establish the trace condition $c_0 = c_0^b $ on $S_1^{\pm}$. We use again the notation $\wea := E_{\eps}(\cea - \ceps^b)  $ (here we use explicity the extension operator $E_{\eps}$ to better distinguish between the extended functions, and the functions itself). We emphasize, that in general we do not have $\wea = 0$ on the whole boundary $S_{\eps}^{\pm}$. However, denoting by $\Oea^{b,\pm}$ the subset of $\Oea$ consisting of micro-cells touching the outer boundary $S_{\eps}^{\pm}$, then we obtain using the standard trace inequality for $\eps$-periodic domains as well as the Poincar\'e-inequlity (use $\wea = 0$ on $\Sea^{\pm}$)
\begin{align*}
    \|\wea\|_{L^2((0,T)\times S_{\eps}^{\pm})} &\le C \left( \frac{1}{\sqrt{\eps}} \|\wea\|_{L^2((0,T)\times \Oea^{b,\pm})}  + \sqrt{\eps} \|\nabla \wea\|_{L^2((0,T)\times \Oea^{b,\pm})}  \right) 
    \\
    &\le C \sqrt{\eps} \|\nabla \wea\|_{L^2((0,T)\times \Oea)}   \le C\eps^{\frac12 - \frac{\alpha}{2}}.
\end{align*}
Since $\alpha<1$, we get the strong two-scale convergence of $\wea$ to $0$ on $S_{\eps}^{\pm}$. Obviously, we have $E_{\eps} \ceps^b \toa c_0^b$, and therefore Proposition \ref{prop:two_scale_trace} implies the desired result.
\end{proof}

It remains to establish the strong two-scale convergence  of $\cea$. First, we consider the case $\Dea = D\mathrm{diag}(\eps^{-\alpha},\ldots,\eps^{-\alpha},\eps^{\alpha})$, which can be seen as the more simple case, since no additional assumptions on the date (see assumption \ref{ass:shift_f}) are necessary. Further, the argument is less technical, because we can apply directly the Simon compactness result from \cite{Simon}.

\begin{proposition}\label{prop:strong_ts_different_scaling} It holds up to a subsequence
\begin{align*}
    \cea \stoa c_0.
\end{align*}
\end{proposition}
\begin{proof}
We define for almost every $(t,x)\in (0,T)\times \Omega$ the rescaled function $\tcea(t,x):= \cea \left(t,\ox,\frac{x_n}{\eps^{\alpha}}\right)$.
With the properties of the extension operator from Lemma \ref{lem:extension_operator} (again, to illustrate the use of this lemma, we explicitly write $E_{\eps} \cea$ for the extended function), we obtain for $0 < h \ll 1$ that  (see \cite[Proposition 5]{gahn2025effective} for similar arguments and more details)
\begin{align*}
\| \tcea (\cdot_t + h,\cdot_x)  - \tcea \|^2_{L^2((0,T-h),L^2(\Omega))} &= \frac{1}{\eps^{\alpha}} \| E_{\eps} \cea (\cdot_t + h , \cdot_x) - E_{\eps} \cea\|^2_{L^2((0,T-h),L^2(\Oea))}
\\
&\le \frac{C}{\eps^{\alpha}} \| \cea (\cdot_t + h , \cdot_x) -  \cea\|^2_{L^2((0,T-h),L^2(\Oeaf))} 
\\
&\le \frac{C\sqrt{h}}{\eps^{\alpha}} \|\partial_t \cea \|_{L^2((0,T),\spaceH_{\Dea}')} \|\cea\|_{L^2((0,T),\spaceH_{\Dea})}
\\
&\le C\sqrt{h},
\end{align*}
where at the end we used the a priori estimates from Proposition \ref{prop:apriori_concentration_diff_nabla} and \ref{prop:apriori_time_derivative}.  Next, we control differences of the shifts in the spatial variable. We extend $\cea$ to a function in $(0,T)\times \R^n$ preserving, in particular, the $L^{\infty}$-estimates.   Hence, due to the essential bound from Lemma \ref{lem:apriori_Linfty}, it is enough to show the above convergence for $\Omega_h:=\left\{x\in \Omega \, : \, \mathrm{dist}(\partial \Omega,x)>h\right\}$ for $0<h \ll 1$. Let $\xi \in \R^n$ such that $|\xi|<h$. In the following, we use similar arguments as in the proof of \cite[Proposition 6]{gahn2025effective} and therefore we skip some details. Let $\bar{\xi}_{\eps} := \eps \left[\frac{\bar{\xi}}{\eps}\right]$. Now, we have
\begin{align*}
 \| \tcea(\cdot_t , \cdot_x + \xi) - \tcea \|_{L^2((0,T) \times \Omega_h)} \le&  \| \tcea(\cdot_t , \cdot_x + (\bar{\xi},0)) - \tcea (\cdot_t , \cdot_x + \xi)\|_{L^2((0,T) \times \Omega_h)}
 \\
 &+  \| \tcea(\cdot_t , \cdot_x + (\bar{\xi},0)) - \tcea (\cdot_t , \cdot_x + (\bar{\xi}_{\eps},0))\|_{L^2((0,T) \times \Omega_h)}
\\
&+  \| \tcea- \tcea (\cdot_t , \cdot_x + (\bar{\xi}_{\eps},0))\|_{L^2((0,T) \times \Omega_h)} =: I_{\eps}^1  + I_{\eps}^2 + I_{\eps}^3.
\end{align*}
For the first term we use the mean value theorem, to obtain
\begin{align*}
    I_{\eps}^1 &\le C |\xi_n| \|\partial_n \tcea\|_{L^2((0,T)\times \Omega)} \le C|\xi_n|.
\end{align*}
For the second term, we proceed in a similar way, to get
\begin{align*}
 I_{\eps}^2 \le C|\bar{\xi} - \bar{\xi}_{\eps}| \| \nabla_{\ox} \tcea\|_{L^2((0,T)\times \Omega)} \le C \eps^{1-\frac{\alpha}{2}} \|\nabla \cea \|_{L^2((0,T)\times \Oeaf)} \le C\eps^{1- \alpha}.
\end{align*}
For the last term $I_{\eps}^3$, we get with Proposition \ref{prop:estimate_shifts} 
\begin{align*}
I_{\eps}^3 \le \frac{C}{\eps^{\frac{\alpha}{2}}} \|\cea (\cdot_t ,  \cdot_x + (\bar{\xi}_{\eps},0)) - \cea \|_{L^2((0,T)\times \Oeaf)} \le C \sqrt{h} + C_h \eps^{\frac{\alpha}{2}} + \kappa(|\bar{\xi}_{\eps}|).
\end{align*}
Now, we can apply \cite[Theorem 1]{Simon} to obtain the strong convergence of $\tcea$ to some limit function $\tilde{c}_0$ in $L^2((0,T)\times \Omega)$. It is easy to check that $\tilde{c}_0 = c_0$. Hence, we get (since $c_0$ is independent of $y$) 
\begin{align*}
\|c_0\|_{L^2(\Omega \times Y)} = \|c_0\|_{L^2(\Omega)} = \lim_{\eps \to 0}\|\tcea\|_{L^2(\Omega)} = \eps^{-\frac{\alpha}{2}} \|\cea \|_{L^2(\Oea)},
\end{align*}
and therefore  the strong two-scale convergence of $\cea$.
\end{proof}
Strong two-scale compactness results in thin layers were also obtained, for example, in \cite{GahnEffectiveTransmissionContinuous,GahnNeussRaduKnabner2018a} (see also \cite{Gahn} for similar ideas in the case of perforated domains). Compared to our situation, a crucial difference lies in the different scaling of the diffusion in different (horizontal and vertical) directions. In the aforementioned contributions, in the case of fast diffusion, an additional bound for the differences of shifts was not necessary. Let us explain, why we cannot avoid this bound in our situation, as long as we consider arbitrary domains. The rescaled (extended) function $\tcea$ fulfills
\begin{align*}
\|\tcea\|_{L^2((0,T),H^1(\Omega))}^2 =& \frac{1}{\eps^{\alpha}} \|E_{\eps}\cea \|_{L^2((0,T)\times \Oea)}^2 + \frac{1}{\eps^{\alpha}} \|\nabla_{\ox} E_{\eps}\cea \|_{L^2((0,T)\times \Oea)}^2\\
&+ \eps^{\alpha}\|\partial_n E_{\eps}\cea \|_{L^2((0,T)\times \Oea)}^2.
\end{align*}
On the right-hand side, we have the norms in the full layer $\Oea$ and have to consider the extended function $E_{\eps}\cea$. Now, as already mentioned at the beginning of this section, we can only control the horizontal gradient $\nabla_{\ox} E_{\eps}\cea $ by the full gradient $\nabla \cea$, and only get
\begin{align*}
    \|\tcea \|_{L^2((0,T)\times H^1(\Omega))} \le C\eps^{-\alpha}.
\end{align*}
This is a consequence of the fact, that the estimate for the gradient for the extension operator is depending on the full norm. We  can only avoid this problem in the case of a specific geometry, for example by considering cylindrical inclusions as in Section \ref{sec:cylindral_inclusions}. In this case, we have
\begin{align*}
  \|\nabla_{\ox} E_{\eps} \cea \|_{L^2((0,T)\times \Oea)} \le C \|\nabla_{\ox} \cea \|_{L^2((0,T)\times \Oeaf)} \le C\eps^{\frac{\alpha}{2}},  
\end{align*}
which implies that $\tcea$ is bounded in $L^2((0,T), H^1(\Omega))$. Now, with the estimate for the differences of the shifts with respect to time in the proof of Proposition \ref{prop:strong_ts_different_scaling}, we can directly apply \cite[Theorem 1]{Simon}, without additional estimate for the differences of the shifts in the spatial variable. However, in conclusion, we see that for an arbitrary shape of the perforations and different orders of diffusion in different directions, the usual a priori bounds in $H^1$ and for the time-derivative, as obtained in Proposition \ref{prop:apriori_concentration_diff_nabla} and \ref{prop:apriori_time_derivative}, are not enough to guarantee strong two-scale convergence. We emphasize, that this problem also occurs in the case of perforated domains which are not thin.

\subsection{Derivation of the macroscopic model}
\label{sec:derivation_macro_model_transport}

Based on the compactness results obtained in the previous section, we are now able for the derivation of the macroscopic model with its effective coefficients. Here, we proceed in the usual way for homogenization problems, and first derive the cell problems for the corrector functions $c_1$, and then derive the macroscopic equation. This has to be adapted to our situation including additional the dimension reduction, which is included in our definition of the two-scale convergence. We first deal with the case \ref{case:diffusion_high} of high diffusion in the horizontal direction. The other case follows by similar arguments. 

\subsubsection{The case $\Dea = D\mathrm{diag}(\eps^{-\alpha},\ldots,\eps^{-\alpha},\eps^{\alpha})$}

First of all, we choose in $\eqref{eq:weak_transport_vea}$ test-functions of the form $\psiea (t,x) :=  \psi\left(t,\ox,\frac{x_n}{\eps^{\alpha}} , \frac{x}{\eps}\right) $ with $\psi \in C_0^{\infty}([0,T), C_{\#}^{\infty}( \overline{\Omega}, C_{\per}^{\infty}(Y)))$ and compact support in $(-1,1)$ with respect to the $x_n$-variable (third component), and get after integration with respect to time and integration by parts in time
\begin{align}
\begin{aligned}\label{eq:aux0_derivation_macro_model}
 - \frac{1}{\eps} &\int_0^T \int_{\Oeaf} \cea  \partial_t\psi\left( t,\ox,\frac{x_n}{\eps^{\alpha}},\frac{x}{\eps}\right) \dd x \dd t 
 \\
 -&  \int_0^T \int_{\Oeaf} \frac{\uea}{\eps^{2}} \cea \cdot \left[ \nabla_{\ox} \psi + \eps^{-\alpha} \partial_{x_n} \psi e_n + \eps^{-1} \nabla_y \psi \right] \left( t,\ox,\frac{x_n}{\eps^{\alpha}},\frac{x}{\eps}\right)\dd x\dd t 
\\
+&  \int_0^T \int_{\Oeaf} \left[\eps^{-\alpha} D\nabla_{\ox} \cea   + \eps^{\alpha} D \partial_n \cea e_n\right]\cdot \left[ \nabla_{\ox} \psi + \eps^{-\alpha} \partial_{x_n} \psi e_n  + \eps^{-1} \nabla_{\bar{y}}\psi\right]\left( t,\ox,\frac{x_n}{\eps^{\alpha}},\frac{x}{\eps}\right) \dd x\dd t
\\
&= \eps^{-\alpha} \int_0^T \int_{\Oeaf} g_{\eps,\alpha }  \psi\left( t,\ox,\frac{x_n}{\eps^{\alpha}},\frac{x}{\eps}\right) \dd x\dd t 
\end{aligned}
\end{align}
First, we derive the cell problem for the limit function $\bar{c}_1$. For this, we multiply the above equation by $\eps$ and 
use the compactness results from Proposition \ref{prop:apriori_concentration_diff_nabla}, to obtain (all terms except the diffusive term including $\nabla_{\oy} \psi$ vanish for $\eps \to 0$)
\begin{align*}
\int_0^T \int_{\Omega} \int_{Y_f} D(\nabla_{\ox} c_0 + \nabla_{\oy} \bar{c}_1) \cdot \nabla_y \psi\dd y  \dd x\dd t = 0.
\end{align*}
By density this is valid for all $\psi \in L^2((0,T)\times \Omega, H_{\per,\nabla_{\oy}}^1(Y_f))$. In other words, $\bar{c}_1$ is the unique (up to a constant) weak solution  of the cell problem
\begin{align}
\begin{aligned}\label{eq:cell_problem_c1_high_diffusion}
    -\nabla_{\oy} \cdot (D(\nabla_{\ox} c_0 + \nabla_{\oy} \bar{c}_1) ) &= 0 &\mbox{ in }& (0,T)\times \Omega \times Y_f,
    \\
    -D(\nabla_{\ox} c_0 + \nabla_{\oy} \bar{c}_1) ) \cdot \nu &= 0 &\mbox{ on }& (0,T)\times \Omega \times \Gamma,
    \\
    \bar{c}_1 \mbox{ is } Y\mbox{-periodic}.
\end{aligned}
\end{align}
From this, we obtain the decomposition 
\begin{align}\label{eq:decomposition_c1_full_gradient}
    \bar{c}_1(t,x,y) = \sum_{i=1}^{n-1} \partial_{x_i} c_0 (t,x)\bchi_i(y)
\end{align}
for almost every $(t,x,y) \in (0,T)\times \Omega \times Y_f$, where $\bchi_i \in H_{\per,\nabla_{\oy}}^1(Y_f)$ is the unique (up to $L^2$-functions only depending on $y_n$) weak solution of the cell problem (for $i=1,\ldots,n-1)$
\begin{align}
\begin{aligned}\label{eq:cell_problem_diffusion}
- \nabla_y \cdot (D (e_i + \nabla_{\oy} \bchi_i)) &= 0 &\mbox{ in }& Y_f,
\\
- D(e_i + \nabla_{\oy}\bchi_i)\cdot \nu &= 0 &\mbox{ on }& \Gamma,
\\
\bchi_i \mbox{ is } Y\mbox{-periodic}.
\end{aligned}
\end{align}
Next, we derive a cell problem for $c_1$. For this, we choose in $\eqref{eq:aux0_derivation_macro_model}$ test-functions independent of $\bar{y}$, i.e., $\phi(t,x,y) = \phi (t,x,y_n)$ and multiply the equation with $\eps^{1-\alpha}$. Now, the term including $\nabla_{\oy} \phi$ vanishes, and we get
\begin{align*}
    \int_{\Omega} \int_{Y_f} (\partial_n c_0 + \partial_{y_n} c_1)\partial_{y_n} \phi (y_n)\dd y  =0,
\end{align*}
and by density this equation is valid for all $\phi \in L^2((0,T)\times \Omega,H_{\per}^1(0,1))$. We define
\begin{align*}
    A(y_n):= \mathcal{H}^{n-1}\left( \left\{ \oy \in Y\, : \, (\oy,y_n) \in Y_f \right\} \right).
\end{align*}
Since $Y_f$ and $\Oeaf$ are Lipschitz and connected, we have $A \in L^{\infty}(0,1)$ and $A\geq a_0>0$. Identifying $c_1$ with a function in $L^2((0,T)\times \Omega , H_{\per}^1(0,1))$, we obtain that $c_1$ is a weak solution of the problem
\begin{align}
\begin{aligned}\label{eq:cell_problem_c_1}
    -\partial_{y_n} (A(\partial_n c_0 + \partial_{y_n}c_1)) &= 0 &\mbox{ in }& (0,T)\times \Omega \times (0,1),
    \\
    c_1 \mbox{ is } 1\mbox{-periodic.}
\end{aligned}
\end{align}
Since this problem has a unique weak solution, we get
\begin{align}\label{eq:decomposition_c1_chi_n}
    c_1(t,x,y_n) = \partial_n c_0(t,x) \chi_n(y),
\end{align}
where $\chi_n \in H_{\per}^1(0,1)/\R$ is the unique weak solution of the cell problem 
\begin{align}
\begin{aligned}\label{eq:cell_problem_chi_n}
- \partial_{y_n} (A (1+ \partial_{y_n} \chi_n )) &= 0 &\mbox{ in }& (0,1),
\\
\chi_n \mbox{ is } 1\mbox{-periodic}.
\end{aligned}
\end{align}
Now, we are able to derive the macroscopic model. In $\eqref{eq:aux0_derivation_macro_model}$ we choose test-functions of the form $\psi(t,x):= \psi\left(t,\ox,\frac{x_n}{\eps^{\alpha}}\right)$ (independent of the microscopic variable $y$), and obtain
\begin{align}
\begin{aligned}\label{eq:aux_derivation_macro_model}
-\eps^{ - \alpha} &\int_0^T \int_{\Oeaf} \cea \partial_t\psi\left( t,\ox,\frac{x_n}{\eps^{\alpha}}\right) \dd x \dd t -  \int_0^T \int_{\Oeaf} \frac{\uea}{\eps^{2}} \cea \cdot \left[ \nabla_{\ox} \psi + \eps^{-\alpha} \partial_{x_n} \psi e_n \right] \left( t,\ox,\frac{x_n}{\eps^{\alpha}}\right)\dd x\dd t
\\
+& \int_0^T \int_{\Oeaf} \left[\eps^{-\alpha} D\nabla_{\ox} \cea   + \eps^{\alpha} D \partial_n \cea e_n\right]\cdot \left[ \nabla_{\ox} \psi + \eps^{-\alpha} \partial_{x_n} \psi e_n \right]\left( t,\ox,\frac{x_n}{\eps^{\alpha}}\right) \dd x\dd t
\\
&= \eps^{-\alpha} \int_0^T \int_{\Oeaf} g_{\eps,\alpha }  \psi\left( t,\ox,\frac{x_n}{\eps^{\alpha}}\right) \dd x\dd t 
\end{aligned}
\end{align}
Using the compactness results from Proposition  \ref{prop:compactness_micro_solution_fluid},  \ref{prop:weak_two_scale_compactness_micro_sol_conc} and \ref{prop:strong_ts_different_scaling} (in particular, we need the strong two-scale convergence of $\cea$ to pass to the limit in the convective term, see also Remark \ref{rem:two_scale_convergence}) , we get 
\begin{align}
\begin{aligned}\label{def:effective_diffusion_coefficient}
-\int_0^T&\int_{\Omega} \int_{Y_f} c_0 \partial_t\psi\dd y  \dd x\dd t  - \int_0^T \int_{\Omega} \int_{Y_f} u_0 c_0 \cdot \partial_{x_n} \psi e_n\dd y  \dd x\dd t
\\
&+ \int_0^T \int_{\Omega} \int_{Y_f} D(\nabla c_0 + (\nabla_{\oy} \bar{c}_1,\partial_{y_n} c_1)) \cdot \nabla \psi\dd y  \dd x\dd t  =  \int_0^T \int_{\Omega} \int_{Y_f} g_0 \psi\dd y  \dd x\dd t.
\end{aligned}
\end{align}
With the decompositions of $\bar{c}_1$ and $c_1$ from $\eqref{eq:decomposition_c1_full_gradient}$ resp. $\eqref{eq:decomposition_c1_chi_n}$,  we get 
\begin{align*}
     \int_{Y_f} D(\nabla c_0 + (\nabla_{\oy} \bar{c}_1,\partial_{y_n} c_1)) \cdot \nabla \psi\dd y = D^{\ast}\nabla c_0 \cdot \nabla \psi 
\end{align*}
almost everywhere in $(0,T)\times \Omega$ with the homogenized diffusion coefficient $D^{\ast} \in \R^{n\times n}$ given by 
\begin{align*}
D_{ij} := \begin{cases}
    \int_{Y_f} D(e_i + \nabla_{\oy} \bar{\chi}_i) \cdot (e_j + \nabla_{\oy} \bar{\chi}_j)\dd y  &\mbox{ for } i,j=1,\ldots,n-1,
    \\
    0 &\mbox{ for } i=n \mbox{ or } j = n,
    \\
    \int_{Y_f} D(1 + \partial_{y_n} \chi_n) (1+\partial_{y_n} \chi_n)\dd y  &\mbox{ for } i=j=n.
\end{cases}
\end{align*}
\begin{remark}
    This formula is also valid in the case $D = \begin{pmatrix}
        \tilde{D} & 0 \\ 0 & D_{nn} \end{pmatrix}$ with $\tilde{D}\in \R^{(n-1)\times (n-1)}$ positive and $D_{nn}>0$.
\end{remark}
Altogether, we end up with 
\begin{align*}
 - |Y_f| \int_0^T \int_{\Omega} c_0 \partial_t \psi& \dd x\dd t - \int_0^T \int_{\Omega} c_0 \ou e_n \cdot \nabla\psi \dd x\dd t 
 \\
 &+ \int_0^T \int_{\Omega} D^{\ast} \nabla c_0 \cdot \nabla \psi \dd x\dd t = \int_0^T \int_{\Omega} \bar{g}_0 \psi \dd x\dd t 
\end{align*}
with $\bar{g}_0:= \int_{Y_f} g_0\dd y $. By density, this is valid for all $\psi \in L^2((0,T),H^1_{\#}(\Omega))$ with $\psi = 0 $ on $S_1^{\pm}$ and $\partial_t \psi \in L^2((0,T)\times \Omega)$. In particular, this implies $\partial_t c_0 \in L^2((0,T),H^1_{\#}(\Omega,S_1^+\cup S_1^-)')$ and we have almost everywhere in $(0,T)$ 
\begin{align*}
    \langle \partial_t c_0 , \psi \rangle_{H^1(\Omega,S_1^+ \cup S_1^-)} - \int_{\Omega} \ou c_0 e_n \cdot \nabla \psi \dd x +\int_{\Omega} D^{\ast} \nabla c_0 \cdot \nabla \psi \dd x = \int_{\Omega} \bar{g}_0 \psi \dd x
\end{align*}
for all $\psi \in H^1_{\#}(\Omega,S_1^+ \cup S_1^-)$ and $c_0(0) = 0$.  In other words, $c_0$ is a weak solution of the macroscopic problem $\eqref{eq:macro_model_transport_high_diffusion}$. Obviously, a weak solution of this problem is unique and, in particular, we get the convergence of the whole sequence. This finishes the proof of Theorem \ref{thm:main_result_transport_high_diffusion}

\subsubsection{The case $\Dea = \eps^{\alpha}D$}

In this case, we proceed in a similar way as before. The only difference occurs in the diffusive term, where this term in $\eqref{eq:aux0_derivation_macro_model}$ has to be replaced by 
\begin{align*}
\int_0^T \int_{\Oeaf} &\eps^{\alpha}D \nabla \cea \cdot \left[ \nabla_{\ox} \psi + \eps^{-\alpha} \partial_{x_n} \psi e_n  + \frac{1}{\eps} \nabla_y \psi\right]\left( t,\ox,\frac{x_n}{\eps^{\alpha}}, \frac{x}{\eps}\right) \dd x\dd t 
\end{align*}
Multiplication with $\eps^{1-\alpha}$ and $\eps\to 0$ gives with the same arguments as above:
\begin{align*}
    \int_0^T \int_{\Omega} \int_{Y_f} D \left( \partial_n c_0 e_n + \nabla_y c_1\right)\cdot \nabla_y  \psi\dd y  \dd x\dd t = 0,
\end{align*}
and, by density, this is valid for all $\phi \in L^2((0,T)\times \Omega,H_{\per}^1(Y_f))$.
In other words, $c_1$ is the unique (up to constant) weak solution of  the cell problem 
\begin{align*}
    -\nabla_y \cdot (D(\partial_n c_0 e_n + \nabla_y c_1) ) &= 0 &\mbox{ in }& (0,T)\times \Omega \times Y_f,
    \\
    -D(\partial_n c_0 e_n + \nabla_y c_1) \cdot \nu &= 0 &\mbox{ on }& (0,T)\times \Omega \times \Gamma,
    \\
    c_1 \mbox{ is } Y\mbox{-periodic}.
\end{align*}
Hence, we obtain for almost every $(t,x,y) \in (0,T)\times \Omega \times Y_f$
\begin{align*}
    c_1(t,x,y) = \partial_n c_0(t,x) \chi_n(y),
\end{align*}
where $\chi_n$ is the cell solution of $\eqref{eq:cell_problem_diffusion}$ for $i=n$. Now, choosing again in $\eqref{eq:weak_transport_vea}$  test-functions of the form $\psi(t,x):= \psi\left(t,\ox,\frac{x_n}{\eps^{\alpha}}\right)$ with $\psi \in C_0^{\infty}([0,T)\times (\Omega \cup \partial_D \Omega))$, we obtain $\eqref{eq:aux_derivation_macro_model}$ with the diffusive term replaced by 
\begin{align*}
    \int_0^T \int_{\Oeaf} &\eps^{\alpha}D \nabla \cea \cdot \left[ \nabla_{\ox} \psi + \eps^{-\alpha} \partial_{x_n} \psi e_n \right]\left( t,\ox,\frac{x_n}{\eps^{\alpha}}\right) \dd x\dd t 
    \\
    \overset{\eps\to 0}{\longrightarrow}& \int_0^T \int_{\Omega} \int_{Y_f} D(\partial_n c_0 e_n + \nabla_y c_1) \cdot e_n \partial_{x_n} \psi\dd y  \dd x\dd t
    \\
    =& \int_0^T \int_{\Omega} D_{nn}^{\ast}  \partial_n c_0 \partial_n \psi \dd x\dd t .
\end{align*}
Arguing as in the previous case, we obtain $\partial_t c_0 \in L^2((0,T)\times \Sigma, H_0^1(-1,1))$, and almost everywhere in $(0,T)$ we have
\begin{align*}
    \langle \partial_t c_0 , \psi \rangle_{L^2(\Sigma,H_0^1(-1,1))} - \int_{\Omega} \ou c_0 \partial_n \psi \dd x +\int_{\Omega} D^{\ast}_{nn} \partial_n c_0 \cdot \partial_n \psi \dd x = \int_{\Omega} \bar{g}_0 \psi \dd x 
\end{align*}
for all $\psi \in L^2(\Sigma,H_0^1(-1,1))$. In other words, $c_0$ is a weak solution of the problem $\eqref{eq:macro_model_transport_low_diffusion}$. It is easy to check that a weak solution of this problem is unique. In particular, all the convergence results are valid for the whole sequence. This finishes the proof of Theorem \ref{thm:main_result_transport_low_diffusion}.

\printbibliography
\end{document}